\documentclass{article}
\usepackage{amsmath,amsfonts,amssymb,amsthm}
\usepackage{color}
\usepackage{hyperref}

\providecommand{\R}{\mathbb{R}}

\providecommand{\N}{\mathbb{N}}

\renewcommand{\leq}{\leqslant}
\renewcommand{\geq}{\geqslant}

\renewcommand{\div}{\operatorname{div}}
\newcommand{\curl}{\operatorname{curl}}

\newcommand{\Id}{\operatorname{Id}}
\newtheorem{Theorem}{Theorem}
\newtheorem{Definition}{Definition}

\newtheorem{Proposition}{Proposition}
\newtheorem{Lemma}{Lemma}

\begin{document}

\date{\today}
\title{On the  ``viscous incompressible fluid + rigid body''  system with Navier conditions}

\author{Gabriela Planas\footnote{Departamento de Matem\'atica, Instituto de  Matem\'atica, Estat\'\i stica  e  Computa\c{c}\~ao Cient\'\i fica,
Universidade Estadual de Campinas, Rua Sergio Buarque de Holanda 651, 13083-859 Campinas, SP, Brasil}
, \ Franck Sueur\footnote{CNRS, UMR 7598, Laboratoire Jacques-Louis Lions, F-75005, Paris, France}
\footnote{UPMC Univ Paris 06, UMR 7598, Laboratoire Jacques-Louis Lions, F-75005, Paris, France}
}

\maketitle

\begin{abstract}
In this paper we consider the motion of a rigid body  in a viscous incompressible fluid when some Navier slip conditions are prescribed on the body's boundary.
The whole system  ``viscous incompressible fluid + rigid body''  is assumed to occupy the full space $\R^{3}$.
We start by proving the existence of global weak solutions to the Cauchy problem.
Then, we exhibit several properties of these solutions.
 First,  we show that the added-mass effect can be computed which yields better-than-expected regularity (in time) of the solid velocity-field.
More precisely we prove that the solid translation and rotation velocities are in the Sobolev space $H^1$.
Second, we show that the case with the body fixed can be thought as the limit of infinite inertia of this system, that is when the solid density  is multiplied by a factor converging to $+\infty$.
Finally we prove the convergence in the energy space  of weak solutions ``\`a la Leray" to smooth solutions of  the system ``inviscid incompressible fluid + rigid body''  as the viscosity goes to zero, till the lifetime $T$  of the smooth solution of the inviscid system.
Moreover we show that the rate of convergence is optimal with respect to the viscosity and that the solid translation and rotation velocities converge in  $H^1 (0,T)$.
\end{abstract}

\section{Introduction}

Recently several efforts have been made to establish a Cauchy theory for various models involving a fluid and a immersed structure.
In particular in the case where the structure is a rigid body and where the fluid is incompressible, there now exists a quite satisfactory range of results, at least in view of what is known in the case of a fluid alone; we may cite, among others,  \cite{ht,Ortega,ogfstt,uniq,Ortega2,rosier,WangZang,shrinking,geodesic,GS,uniq,GS3} for the case of inviscid fluid  and \cite{gm,hs,Serre,DE1,DE2,Conca1,Conca2,FE1,FE2,FE3} for the case of viscous fluid.

\par
\  \par

In this paper we deal with the issue of the inviscid limit for the system ''viscous incompressible fluid + rigid body", which involves an immersed rigid body moving into a viscous incompressible fluid driven by the Navier-Stokes equations.

Formally dropping the viscosity in the equations of the system yields the system ``inviscid incompressible fluid + rigid body", which involves an immersed rigid body moving into a inviscid incompressible fluid driven by the Euler equations.

\par
\  \par
This paper aims at giving a justification of this formal procedure.

It is expected that the issue of justifying the inviscid limit is at least as difficult in the case of a moving body as in the case of a fixed body.
Indeed the case of a fixed body can be seen as the limit case where the body's inertia becomes infinite, cf. Sections \ref{SectionInertiaNS} and  \ref{SectionInertiaEuler}.

This is quite a bad news since the inviscid limit is already quite intricate in the case of a fixed boundary, because of the boundary layers phenomenon.
Moreover it is not clear a priori if it is possible to pass to the limit in the body's dynamics, with such singular variations in the neighborhood.

\par
\  \par

These boundary layers issues are particularly involved in the case where one prescribes the no-slip condition on a fixed fluid-solid interface.
In particular, it is not known at the time of writing if there is convergence in the energy space.
A longstanding approach in this domain is Prandtl's theory, but this theory fails to model flows with too small viscosity in general, see for instance \cite{constantin,BardosTiti,E,grenier,DGVD,guo}.
However a necessary and sufficient condition for  the convergence to the Euler equations  in the energy space has been given by Kato in \cite{Tosio}; it states that the energy dissipation rate of the viscous  flows in a boundary strip  of width proportional to the viscosity vanishes.
Even if little is known about whether or not this condition is verified for a general given flow, this result gives a nice insight of the scale for which the description of the flow is necessary to understand the inviscid limit.

In the paper \cite{KatoBody} the second author proved an extension of Kato's result in
the case of the system ``viscous incompressible fluid + rigid body''  with  the no-slip condition on the fluid-solid interface: the convergence to the  system ``inviscid incompressible fluid + rigid body"  holds true in the energy space if and only if the energy dissipation rate of the viscous  flows in a boundary's neighborhood  of width proportional to the viscosity vanishes.
As in the case of a fixed boundary it is not clear how to check this condition but this result seems to indicate that the issue of the inviscid limit in the case of a moving body is maybe not much harder than the one in the case of a fixed body.

\par
\  \par

In this paper we will  prescribe some Navier conditions on the fluid-solid interface, which   encode that
the fluid slips with some friction on this boundary.
In the case of a fixed boundary it is by now well understood that the issue of the inviscid limit is simplified compared to the no-slip conditions, at least when the friction coefficient is not too big;
in particular the convergence holds true in the energy space (cf. for instance \cite{BGP,dV,cmr,Iftimie-Planas,Iftimie-Sueur,MasmoudiRousset,Paddick}).
We prove here a similar result in the case of a moving body.

\par
\  \par

In fact we even prove a better convergence of the body's dynamics (with respect to the fluid's dynamics), i.e. a  convergence in the Sobolev space $H^1$.
This surprising  result uses a well-known phenomenon in the theory of the systems involving an incompressible flow and a structure, namely the added mass phenomenon, for which we refer  for instance to \cite{Childress,Galdi}, and which can be computed in the present case of the Navier conditions.

\par
\  \par

Let us say here for sake of clarity that we will consider the case of a physical space of three dimensions and we will assume that the system occupies the whole of $\R^3$ to avoid the extra difficulties which would be implied by an exterior boundary.

After finishing this paper we became aware of the work \cite{DGV-H} by G\'erard-Varet and Hillairet, which  established the existence of weak solutions to the ``viscous incompressible fluid + rigid body''  system with Navier slip conditions in the case where the  whole system occupies a bounded domain of $\R^{3}$, rather than the full space $\R^{3}$, up to collision. It would be therefore interesting to look for some extensions of the properties exhibited here in such a case.

\par
\  \par

The paper is organized as follows.
\begin{itemize}
\item In Section \ref{NavierNS} we introduce the  system ``viscous incompressible fluid + rigid body'' with Navier conditions.
\item In Section \ref{NavierLeray} we establish the existence of an appropriate notion of weak solutions ``\`a la Leray" of this system, after  a change of variables. We will also establish a regularity property of the body's dynamics  and we will  discuss the infinite inertia limit for this system.

\item In Section \ref{Sl} we recall a result of \cite{KatoBody} which establishes the existence of  smooth local-in-time solutions of the inviscid system and  we also discuss the infinite inertia limit.

\item In Section \ref{IL} we state the main result of this paper about the convergence of the system ``viscous incompressible fluid + rigid body''  to  the system ``inviscid incompressible fluid + rigid body''  as the viscosity goes to zero.
\item Finally we give the proof of this result in Section
\ref{proof}.
\end{itemize}


\section{The system ``viscous incompressible fluid + rigid body'' with Navier conditions}
\label{NavierNS}

We consider a rigid body initially occupying  a closed, bounded, connected and simply connected subset $\mathcal{S}_0 \subset \R^3$ with smooth boundary.
It rigidly moves so that at  time $t$  it occupies an isometric  domain denoted by $\mathcal{S}(t)$.
More precisely if we denote by $h (t)$ the position of the center of mass of the body at time $t$, then
there exists a rotation matrix $Q (t) \in SO(3)$,
such that  the position $\eta (t,x) \in \mathcal{S} (t)$  at
the time $t$ of the point fixed to the body with an initial position $x$ is
\begin{equation*}
\eta (t,x) := h (t) + Q (t)(x- h (0)) .
\end{equation*}
Of course this yields that $Q(0) = \Id_3$.

Moreover since  $Q^{T} Q' (t) $ is skew symmetric there exists only one $r  (t)$ in $\R^3$ such that for any $x \in \R^3$,
\begin{equation*}
Q^{T} Q' (t)  x = r(t) \wedge x .
\end{equation*}
Accordingly, the solid velocity is given by
\begin{equation*}
U_{{\mathcal S}}(t,x) := h'(t)  + R(t) \wedge (x-h(t)) \text{ with } R(t) := Q(t)r(t) .
\end{equation*}
Given a positive function $\rho_{{\mathcal S}_{0}} $, say in $ L^{\infty}({\mathcal S}_{0};\R)$, describing the density in the solid initially:
 the solid mass $m>0$, the initial position $h_0$  of the center of mass, and the initial value of  the  inertial matrix ${\mathcal J}_0$ can be computed by it first moments:
\begin{equation} \label{EqMasse}
m :=  \int_{\mathcal{S}_0} \rho_{{\mathcal S}_0} (x) dx  > 0,
\end{equation}
\begin{equation} \label{Eq:CG}
m h_0 :=   \int_{\mathcal{S}_0} x \rho_{{\mathcal S}_0} (x) dx,
\end{equation}
\begin{equation}\label{eqJ}
\mathcal{J}_{0}  := \int_{ \mathcal{S}_{0}}  \rho_{{\mathcal S}_{0}}(x)   \big( | x- h_0 |^2 \Id_3 -(x- h_0)  \otimes   (x- h_0)    \big) dx  .
\end{equation}
At time time $t >0$, the density in the solid is given, for $x \in \mathcal{S}(t)$, by
\begin{equation*}
\rho_{{\mathcal S}} (t,x) := \rho_{{\mathcal S}_0} (\eta (t,x)^{-1} (x) ) ,
\end{equation*}
where $\eta (t,x)^{-1} $ denotes the inverse at time $t$ of the diffeomorphism $ x \mapsto  \eta(t, x)$;
so that, of course, the mass is preserved:
\begin{equation*}
m =  \int_{\mathcal{S} (t)} \rho_{{\mathcal S}} (t,x) dx  .
\end{equation*}
Moreover, the position of the center of mass $h(t)$ and the  inertial matrix ${\mathcal J} (t)$ are given by
\begin{equation*}
m h (t) :=   \int_{ \mathcal{S} (t) } x \rho_{{\mathcal S}} (t,x) dx,
\end{equation*}
\begin{equation*}
\mathcal{J} (t) := \int_{ \mathcal{S} (t)}  \rho_{{\mathcal S}} (t,x)   \big( | x- h (t) |^2 \Id_3 - (x- h (t))  \otimes   (x- h (t))    \big) dx  ,
\end{equation*}
so that   $ {\mathcal J} (t)$ is symmetric positive definite and satisfies  Sylvester's law:
\begin{equation*}
 \mathcal{J}(t) = Q(t)  \mathcal{J}_0 Q^{T} (t) .
 \end{equation*}
\par
\  \par

Let us assume that in the rest of the space, that is, in the open set
$\mathcal{F}(t) := \R^3 \setminus {\mathcal S} (t)$,
there evolves a viscous incompressible  fluid.
We denote correspondingly ${\mathcal F}_{0}:=\R^{3} \setminus {\mathcal S}_{0}$ the initial fluid domain. \par
The complete system driving the dynamics reads
\begin{gather}
\label{NS1}
\frac{\partial U}{\partial t}+(U\cdot\nabla)U + \nabla P = \nu  \Delta U  \ \text{ for } \ x \in \mathcal{F}(t), \\
\label{NS2}
\div U = 0 \ \text{ for } \  x \in \mathcal{F}(t) ,  \\
\label{NS3}
U \cdot n =  U_\mathcal{S}\cdot n   \ \text{ for } \ x\in \partial \mathcal{S}  (t),  \\
\label{NS4}
(D(U) n )  \wedge n = - \alpha (U - U_\mathcal{S}) \wedge n   \ \text{ for } \ x\in \partial \mathcal{S}  (t),  \\
\label{Solide1}
m  h'' (t) =  - \int_{ \partial \mathcal{S} (t)} \Sigma n \, ds ,  \\
\label{Solide2}
(\mathcal{J}  R )' (t) = -    \int_{ \partial   \mathcal{S} (t)}  (x-  h(t) )  \wedge \Sigma  n \, ds , \\
\label{NSci2}
U |_{t= 0} = U_0 , \\
\label{Solideci}
h (0)=  0 , \ h' (0)= \ell_0, \ R  (0)=  r_0.
\end{gather}
Here  $U$ and  $P$  denote the fluid velocity and pressure, which are   defined on ${\mathcal F}(t)$ for each $t$, and $ \nu > 0 $ is the fluid viscosity.
  The fluid is supposed to be  homogeneous of density $1$, to simplify the notations and without any loss of generality.
The Cauchy stress tensor is defined by
\begin{equation*}
 \Sigma := -P \Id_3 + 2 \nu D(U) ,
\end{equation*}
where $D(U) $ is the deformation tensor
\begin{equation*}
 D(U) := ( \frac{1}{2} ( \partial_{j} U_{i} +  \partial_{i} U_{j} ) )_{1 \leqslant i,j \leqslant  3} .
\end{equation*}
Above $n$ denotes the unit outward normal on the boundary of the fluid domain,
$ds$ denotes the integration element on this boundary
 and $\alpha \geq 0 $ is a material constant (the friction
coefficient).
Let us precise that we choose here to consider the case where  $\alpha $ is constant but it will be possible to consider the more general case where $\alpha $ depends smoothly on $t,x$ as well with only a few modifications.

Observe that, without loss of generality, we have assumed that $h(0)=0$ which means that the body is centered at the origin at the initial time $t=0$.

\
In the sequel the integrals over open subsets of $\R^{3}$ will always be  taken with respect to the Lebesgue measure $dx$ and the integrals over hypersurfaces of $\R^{3}$ will always be  taken with respect to the  surface measure.
\par
\  \par

The equations  \eqref{NS1}-\eqref{NS2} are the incompressible Navier-Stokes equations.

The equations  (\ref{Solide1})  and (\ref{Solide2})  are  the Newton's balance laws for linear and angular momenta: the fluid acts on the body through pressure forces.

The equations  \eqref{NS3}-\eqref{NS4} are referred to as the Navier conditions and encode that  the body's boundary is impermeable and that
the fluid slips with some friction on this boundary.

This condition was introduced phenomenologically by Navier in $1823$, cf.
\cite{Navier}.
Let us mention some recent results about the derivation of such a condition, on the one hand from
kinetic models (derivation from the Boltzmann equation with accommodation boundary conditions)
see \cite{coco,MasmoudiLSR,BGP}, and on the other hand from homogenization of rough boundaries
\cite{dadgv,dgvmasmoudi}.

\section{Weak solutions ``\`a la Leray"}
\label{NavierLeray}

In this section we start by use the change of variables introduced by Serre in \cite{Serre} in order to fix the fluid domain.
It is fair to point out that this change of variable is here particularly simple as there is no exterior boundary.
Then we establish the existence of an appropriate notion of weak solutions ``\`a la Leray" of this system.
We will observe that  the solid velocity benefits from extra regularity with respect to what is expected from the energy estimate.
We will also discuss the infinite inertia limit.

\subsection{A change of variables} \label{changeNS}

In order to write the equations of the fluid in a fixed domain, we are going to use the following changes of variables:
\begin{equation*}
\ell(t) :=   Q(t)^T \,   h' (t)  ,  \, u(t,x) := Q(t)^T \,  U(t, Q(t) x+h(t)) ,
\ p(t,x) := P(t, Q(t) x+ h(t))  ,
\end{equation*}
and we introduce
\begin{equation*}
 \sigma := -p \Id_3 + 2 \nu D(u) , \text{ where  }  D(u) := ( \frac{1}{2} ( \partial_{j} u_{i} +  \partial_{i} u_{j} ) )_{i,j} .
\end{equation*}
Therefore the system \eqref{NS1}-\eqref{Solideci}  now reads
\begin{gather}
\label{chNS1}
\frac{\partial u}{\partial t}+ (  u - u_\mathcal{S} ) \cdot \nabla u +  r \wedge u  + \nabla p =  \nu \Delta u \ \text{ for } \ x \in \mathcal{F}_0 , \\
\label{chNS2}
\div u = 0 \ \text{ for } \  x \in \mathcal{F}_0  ,  \\
\label{chNS3}
u  \cdot n = u_\mathcal{S} \cdot n  \ \text{ for } \ x\in \partial \mathcal{S}_0,  \\
\label{chNS4}
(D(u) n )  \wedge n = - \alpha (u - u_\mathcal{S}) \wedge n   \ \text{ for } \ x\in \partial \mathcal{S}_0 ,  \\
\label{chSolide1}
m  \ell'  =  - \int_{ \partial \mathcal{S}_0 } \sigma n \, ds  + m \ell \wedge  r,  \\
\label{chSolide2}
\mathcal{J}_0 r'  =   - \int_{ \partial   \mathcal{S}_0}   x \wedge \sigma n \, ds  + ( \mathcal{J}_0  r) \wedge r , \\
\label{chNSci2}
u |_{t= 0} = u_0 , \\
\label{chSolideci}
h (0)= 0  , \ h' (0)= \ell_0 ,\ r  (0)=  r_0,
\end{gather}
with
\begin{equation}
\label{Svelo}
  u_{\mathcal{S}} (t,x) := \ell  (t)  + r(t) \wedge x .
   \end{equation}

\subsection{A weak formulation} \label{weaksolution}

With of the purpose of writing a weak formulation of the system \eqref{chNS1}-\eqref{chSolideci} we introduce the following space
\begin{equation*}
 \mathcal{H} := \{ \phi \in L^{2} (\R^{3}) / \ \div \phi = 0  \text{ in }  \R^{3} \text{ and } D(\phi) = 0  \text{ in }  \mathcal{S}_0 \} .
\end{equation*}
According to Lemma $1.1$ in \cite{Temam}, p.18,  for all  $\phi \in  \mathcal{H}$, there exist $\ell_{\phi} \in \R^{3}$ and $r_{\phi} \in \R^{3} $ such that for any $x \in  \mathcal{S}_0$,
$\phi (x) = \ell_{\phi} + r_{\phi} \wedge  x $.
Therefore we extend the initial data $u_{0}$  by setting
$u_{0} := \ell_{0} + r_{0} \wedge x $  for $x \in \mathcal{S}_0 $.

Conversely, when $\phi \in  \mathcal{H}$, we denote  $\phi_{\mathcal{S}}$ its restriction to $\mathcal{S}_0 $.
\par
\  \par

Let us give here a result which will be useful in the sequel.
\begin{Lemma}
\label{Useful}
For any  $u,v \in \mathcal{H}$ with $u\vert_{\mathcal{F}_0} \in H^2$ and $v\vert_{\mathcal{F}_0} \in H^1$,
\begin{eqnarray*}
 \int_{ \mathcal{F}_0}  \Delta u \cdot v &=& - 2 \int_{ \mathcal{F}_0} D(u) :  D(v)
 + 2 \ell_{v} \cdot \int_{ \partial \mathcal{S}_0 } D(u) n \, ds    + 2 r_{v}  \cdot  \int_{ \partial \mathcal{S}_0 }  x \wedge D(u) n \, ds
  \\ &&+ 2 \int_{ \partial \mathcal{S}_0}    \Big( (D(u)  n)  \wedge n \Big)  \cdot \Big(  (v-v_\mathcal{S} ) \wedge  n \Big) .
\end{eqnarray*}
\end{Lemma}
\begin{proof}
We have
\begin{equation*}
 \int_{ \mathcal{F}_0}  \Delta u \cdot v =  2 \int_{ \partial \mathcal{S}_0}  (D(u)  v) \cdot  n - 2 \int_{ \mathcal{F}_0} D(u) :  D(v)
 =  2 \int_{ \partial \mathcal{S}_0}  (D(u)  n) \cdot v - 2 \int_{ \mathcal{F}_0} D(u) :  D(v) ,
\end{equation*}
since $D(u) $ is symmetric.
Moreover
\begin{equation*}
 \int_{ \partial  \mathcal{S}_0}   (D(u)  n) \cdot v
 =   \int_{ \partial  \mathcal{S}_0}  \Big( (D(u)  n)  \cdot n \Big) \Big(  v \cdot  n \Big)
+ \int_{ \partial  \mathcal{S}_0} \Big( (D(u)  n)  \wedge n \Big)  \cdot \Big(  v \wedge  n \Big) .
  \end{equation*}
But
\begin{eqnarray*}
 \int_{ \partial  \mathcal{S}_0}  \Big( (D(u)  n)  \cdot n \Big) \Big(  v \cdot  n \Big)
 &=&  \int_{ \partial  \mathcal{S}_0}  \Big( (D(u)  n)  \cdot n \Big) \Big(  v_\mathcal{S}  \cdot  n \Big)
\\ &=&  \int_{ \partial  \mathcal{S}_0}   \Big( \big( (D(u)  n)  \cdot n \big)  n \Big)   \cdot v_\mathcal{S}
\\ &=&  \int_{ \partial  \mathcal{S}_0}   \Big( D(u)  n -  (D(u)  n)_{tan}   \Big)   \cdot v_\mathcal{S}
\\ &=&   \ell_{v} \cdot \int_{ \partial \mathcal{S}_0 } D(u)  n \, ds    + r_{v}  \cdot  \int_{ \partial \mathcal{S}_0 }  x \wedge D(u)  n \, ds
-  \int_{ \partial  \mathcal{S}_0}  (D(u)  n)_{tan} \cdot v_\mathcal{S}
\\ &=&   \ell_{v} \cdot \int_{ \partial \mathcal{S}_0 } D(u)  n \, ds    + r_{v}  \cdot  \int_{ \partial \mathcal{S}_0 }  x \wedge D(u)  n \, ds
-  \int_{ \partial  \mathcal{S}_0}  \Big(D(u)  n \wedge  n \Big)\cdot \Big(v_\mathcal{S}\wedge  n \Big) .
\end{eqnarray*}
Above we use the index ``tan" to denote the tangential part of a vector field defined on $\partial \mathcal{S}_0$.
 Gathering the previous identities yields the result.
\end{proof}

Now we endow the space $L^{2} (\R^{3})$ with the following inner product, which is equivalent to the usual one,
\begin{equation*}
(\phi  , \psi  )_{\mathcal{H}} :=    \int_{\mathcal{F}_0 }  \phi   \cdot   \psi  \,  dx   +  \int_{\mathcal{S}_0 }  \rho_{S_{0}}  \phi   \cdot   \psi  \, dx .
\end{equation*}
When $\phi  , \psi $  are in  $\mathcal{H}$ we obtain:
\begin{equation}
\label{plusimple}
(\phi  , \psi  )_{\mathcal{H}} =    \int_{\mathcal{F}_0 }  \phi \cdot  \psi  \, dx  +  m  \ell_{\phi} \cdot  \ell_{\psi} + \mathcal{J}_0  r_{\phi}   \cdot   r_{\psi}  ,
\end{equation}
by using \eqref{EqMasse}-\eqref{Eq:CG}-\eqref{eqJ}. The spaces $L^{2} (\R^{3})$ and $ \mathcal{H} $ are clearly Hilbert spaces for the scalar product $(\cdot,\cdot)_{\mathcal{H}}$.

\begin{Proposition}
\label{WeakStrong}
A smooth solution $u$ of \eqref{chNS1}-\eqref{chSolideci} satisfies the following: for any $v \in C^{\infty } ([0,T ] ;  \mathcal{H} )$ such that $v |_{\overline{\mathcal{F}_0}} \in C^{\infty } ([0,T ] ;  C^{\infty }_{c} (\overline{\mathcal{F}_0} ))$,
for all $t\in [0,T ] $,
\begin{equation}
\label{WeakNS}
(u,v)_{\mathcal{H}}  (t) -  (u_{0},v |_{t=0})_{\mathcal{H}}   =
 \int_{0}^{t} \Big[ (u, \partial_{t} v)_{\mathcal{H}}
 + 2  \nu a(u,v) +  b(u,u,v)    \Big] ,
\end{equation}
where
\begin{eqnarray*}
a(u,v) &:=& - \alpha  \int_{\partial \mathcal{S}_0} (u - u_\mathcal{S} ) \cdot (v - v_\mathcal{S} ) - \int_{\mathcal{F}_0 }     D(u) :  D(v)
\\ b(u,v,w) &:=& m  \det (r_u ,\ell_{v} ,  \ell_w ) + \det (\mathcal{J}_0 r_u , r_{v} , r_w ) +
\int_{\mathcal{F}_0 } \Big(     [   (u-  u_\mathcal{S} )  \cdot\nabla w ] \cdot v   -  \det (r_u , v ,  w ) \Big).
\end{eqnarray*}
Moreover this solution $u$ satisfies the following energy equality: for almost any  $t \in [0,T ]$,
\begin{equation}
\label{NSBodyWeakEnergyEqu}
\frac{1}{2} \| u (t, \cdot)  \|^{2}_{\mathcal{H}} + 2 \nu \int_{(0,t) \times \mathcal{F}_0   } | D(u) |^{2}   + 2\alpha\nu \int_0^t \int_{ \partial \mathcal{S}_0} |u- u_\mathcal{S}|^2
= \frac{1}{2}   \| u_{0}  \|^{2}_{\mathcal{H}}     .
\end{equation}
\end{Proposition}
In the sequel we will consider several asymptotics with respect to the parameters $\alpha, \nu, m$ and $ \mathcal{J}_0  $.
Let us therefore stress here that the forms $a$ and $b$ above depend respectively on  $\alpha ,  \mathcal{S}_0  $ and on $m,\mathcal{J}_0 , \mathcal{S}_0 $.
\begin{proof}
Let  $u$ be a smooth solution of \eqref{chNS1}-\eqref{chSolideci} on $[0,T ] $ and $v \in C^{\infty } ([0,T ] ;  \mathcal{H} )$ such that $v |_{\overline{\mathcal{F}_0}} \in C^{\infty } ([0,T ] ;  C^{\infty }_{c} (\overline{\mathcal{F}_0} ))$.
We first  observe  that the result of Proposition \ref{WeakStrong} will follow, by an integration by parts in time, from the following claim: for any $t \in [0,T]$,
\begin{equation}
\label{WeakNSDos}
  (\partial_{t} u, v)_{\mathcal{H}}
   =  2\nu a(u,v)  + b(u,u,v)   .
\end{equation}
To prove the claim,  we multiply  equation \eqref{chNS1} by $v$ and integrate over $\mathcal{F}_0$:
\begin{equation*}
 \int_{ \mathcal{F}_0}  \frac{\partial u}{\partial t} \cdot v +  \int_{ \mathcal{F}_0}  [ \big( (u- u_\mathcal{S}    ) \cdot\nabla \big)u  ]\cdot v
 +  \int_{ \mathcal{F}_0}  ( r(t) \wedge u  )  \cdot v   +  \int_{ \mathcal{F}_0}  \nabla p  \cdot v
    =  \int_{ \mathcal{F}_0} \nu \Delta u \cdot v    .
\end{equation*}
We then use some integrations by parts, taking into account \eqref{chNS2} and \eqref{chNS3}, to get
\begin{eqnarray*}
  \int_{ \mathcal{F}_0} [ \big( (u- u_\mathcal{S}    ) \cdot\nabla \big)u  ]\cdot v  &=& - \int_{ \mathcal{F}_0} u \cdot  \big( (u-  u_\mathcal{S}  ) \cdot\nabla v \big) ,
 \\  \int_{ \mathcal{F}_0}   ( r(t) \wedge u  )  \cdot v &=&   \int_{ \mathcal{F}_0}  \det (r,u,v)   ,
 \\  \int_{ \mathcal{F}_0}  \nabla p  \cdot v &=&   \int_{ \partial  \mathcal{S}_0}  p  n \cdot v  .
\end{eqnarray*}
Next, we observe that
\begin{equation*}
 \int_{ \partial  \mathcal{S}_0}  p  n \cdot v
 = \ell_{v} \cdot \int_{ \partial \mathcal{S}_0 } p n \, ds   +  r_{v}  \cdot  \int_{ \partial \mathcal{S}_0 }  x \wedge p n \, ds .
  \end{equation*}
   Therefore, using Lemma \ref{Useful}, the Navier conditions and \eqref{chSolide1}-\eqref{chSolide2}, we obtain
\begin{eqnarray*}
\int_{ \partial  \mathcal{S}_0}  p  n \cdot v -\int_{ \mathcal{F}_0} \nu \Delta u \cdot v
   &=& - \ell_{v} \cdot \int_{ \partial \mathcal{S}_0 } \sigma n \, ds    - r_{v}  \cdot  \int_{ \partial \mathcal{S}_0 }  x \wedge \sigma n \, ds
  \\ && + 2 \alpha \nu\int_{ \partial  \mathcal{S}_0}  (u - u_\mathcal{S} )   \cdot   (v-v_\mathcal{S} )  + 2 \nu\int_{ \mathcal{F}_0} D(u) :  D(v)
\\ &=& m  \ell_{v}  \cdot   \ell'   + \mathcal{J}_0    r_{v}  \cdot r'  -  \det (m\ell,r, \ell_{v})  -  \det (  \mathcal{J}_0  r ,r,    r_{v})
\\ && + 2 \alpha \nu\int_{ \partial  \mathcal{S}_0}  (u - u_\mathcal{S} )   \cdot   (v-v_\mathcal{S} )  + 2 \nu \int_{ \mathcal{F}_0} D(u) :  D(v) .
\end{eqnarray*}
Gathering all these equalities yields \eqref{WeakNSDos}.
Finally the  energy equality \eqref{NSBodyWeakEnergyEqu} follows from \eqref{WeakNS} by specifying the test function as $v=u$.
\end{proof}
Let us recall that according to Korn's inequality, see for instance \cite{Mclean}[Th. 10.2], the  energy equality \eqref{NSBodyWeakEnergyEqu} yields that $u \in L^2 ( 0,T  ;  \underline{\mathcal{V}}  )$, where $ \underline{\mathcal{V}} $ is given by
\[\underline{\mathcal{V}}:=   \{  \phi \in   \mathcal{H} / \ \int_{ \mathcal{F}_0 } | \nabla \phi  (y) |^2  dy < + \infty    \} \ \text{  with  norm } \ \ \|  \phi  \|_{\underline{\mathcal{V}}} := \|  \phi  \|_\mathcal{H} + \|  \nabla  \phi  \|_{L^{2} (\mathcal{F}_0 ,dy )   }. \]

We introduce now the concept of weak solutions ``\`a la Leray".
\begin{Definition}[]
\label{Weak}
We say that
\begin{equation*}
u \in C_w ( [0,T ]  ; \mathcal{H} ) \cap L^2 ( 0,T  ;  \underline{\mathcal{V}}  )
\end{equation*}
is a weak solution of the system \eqref{chNS1}-\eqref{chSolideci}
if for all $v \in C^{\infty } ([0,T ] ;  \mathcal{H} )$ such that $v |_{\overline{\mathcal{F}_0}} \in C^{\infty } ([0,T ] ;  C^{\infty }_{c} (\overline{\mathcal{F}_0} ))$
and for all $t\in [0,T ] $, \eqref{WeakNS} holds true.
\end{Definition}
Let us remark that a standard density argument allows us to take less smooth test vector
fields $v$ in the above weak formulation.  More precisely, to enlarge the space of the test functions, we introduce the space
\begin{equation*}
\mathcal{V} :=   \{  \phi \in   \mathcal{H} / \ \int_{ \mathcal{F}_0 } | \nabla \phi  (y) |^2 (1 + | y |^2 ) dy < + \infty    \} ,
\end{equation*}
endowed with the norm
\begin{equation*}
 \|  \phi  \|_\mathcal{V} := \|  \phi  \|_\mathcal{H} + \|  \nabla  \phi  \|_{L^{2} (\mathcal{F}_0  ,(1 + | y |^2 )^{\frac{1}{2}} dy ) }  .
\end{equation*}
It is worth to notice from now on that $b$ is a  trilinear  continuous form on $\underline{\mathcal{V}} \times \underline{\mathcal{V}} \times \mathcal{V}$: there exists a constant $C>0$ such that
for any $(u,v,w) \in \underline{\mathcal{V}} \times  \underline{\mathcal{V}} \times \mathcal{V}$,
\begin{equation}
\label{bconti}
|b(u,v,w)  |  \leqslant C \| u  \|_{\underline{\mathcal{V}}} \,  \| v  \|_{\underline{\mathcal{V}}} \,  \| w \|_{\mathcal{V}} .
\end{equation}
This follows easily from Holder's inequality and the following interpolation inequality
\begin{equation}
\label{interpo}
\|  v \|_{L^4 (\mathcal{F}_{0})} \leq \sqrt{2}  \|  v \|^{\frac{1}{4}}_{L^2 (\mathcal{F}_{0})} \|  \nabla v \|^{\frac{3}{4}}_{L^2 (\mathcal{F}_{0})} .
\end{equation}
Observe in particular that the weight in the definition of  $\mathcal{V}$ allows to handle the rotation part of $u_\mathcal{S}$.

Moreover the trilinear form $b$ satisfies the following crucial property
\begin{equation}
\label{Tonga}
(u,v) \in \underline{\mathcal{V}} \times \mathcal{V} \text{ implies } b(u,v,v) = 0 .
\end{equation}

On the other hand, for any $u,v$ in $\underline{\mathcal{V}} $,
\begin{equation}
\label{Conta}
| a(u,v)  |  \leq  C   \|  u \|_{\underline{\mathcal{V}}} \,    \|  v \|_{\underline{\mathcal{V}}}   .
\end{equation}
In fact, to deal with the boundary integral, we introduce a smooth cut-off function $\chi$ defined on $\overline{\mathcal{F}_0}$ such that $\chi = 1$ in $\Gamma_{c}$ and  $\chi = 0$ in $\mathcal{F}_0 \setminus \Gamma_{2 c}$, where
\begin{equation*}
\Gamma_{c} := \{  x \in  \mathcal{F}_0 / \  d(x)  < c  \} \text{ with }  d(x) := dist (x,  \partial   \mathcal{S}_0 ) .
\end{equation*}
Let us denote
\begin{equation*}
    \psi^u_{\mathcal{S}} (t,x) :=  \frac12 (  \ell_u (t) \wedge  x  - r_u  (t)  | x |^{2} )   \text{ and } \tilde{u}_{\mathcal{S}} := \curl ( \chi  \psi^u_{\mathcal{S}} ) ,
\end{equation*}
and let us define similarly $\tilde{v}_{\mathcal{S}}$.

Thus, $\tilde{u}_{\mathcal{S}}$ and $\tilde{v}_{\mathcal{S}}$ are equal respectively, to $u_{\mathcal{S}}$ and ${v}_{\mathcal{S}}$ near $ \partial \mathcal{S}_0$, vanish away $ \mathcal{S}_0 $, and are divergence free. Moreover, $ \|\tilde{u}_{\mathcal{S}}\|_{H^1(\mathcal{F}_0)} \leq C (\|\ell_u\| + \|r_u\|)$, and a similar estimate holds for $\tilde{v}_{\mathcal{S}}$.  Then, we apply the H\"older inequality and the trace theorem, to arrive at
\begin{align*} \Bigl|\int_{\partial \mathcal{S}_0} (u - u_\mathcal{S} ) \cdot (v - v_\mathcal{S} )\Bigr| & = \Bigl| \int_{\partial \mathcal{S}_0} (u - \tilde{u}_\mathcal{S} ) \cdot (v - \tilde{v}_\mathcal{S} )\Bigr| \\
& \leq C \|u - \tilde{u}_\mathcal{S}\|_{H^1(\mathcal{F}_0)}   \|v - \tilde{v}_\mathcal{S}\|_{H^1(\mathcal{F}_0)} \\
& \leq C (  \| u\|_{H^1(\mathcal{F}_0)} + \|\ell_u\| + \|r_u\| )  (  \| v\|_{H^1(\mathcal{F}_0)} + \|\ell_v\| + \|r_v\| ),
\end{align*}
so that, \eqref{Conta} follows.

These previous arguments allows us to take less smooth test vector
fields $v$ in the weak formulation \eqref{WeakNS}, for instance, belonging to $ H^1 ( 0,T   ; \mathcal{H} ) \cap  L^4 (0,T  ;  \mathcal{V} )$.

Finally let us mention that to a weak solution we may associate a pressure such that the equations are satisfied in the distribution sense, and prove that a regular weak solution is a solution in the classical sense; following for instance \cite{Serre}, Section $III$ and \cite{TemamNS} with a few straightforward adaptations.

\subsection{An extension of Leray's theorem}
\label{Leray}

The following result establishes the existence of global weak solutions of the system \eqref{chNS1}-\eqref{chSolideci}.
\begin{Theorem}
\label{NSBodyWeak}
Let be given $u_{0} \in  \mathcal{H}$  and $T >0$.
Then there exists a weak solution $u$ of \eqref{chNS1}-\eqref{chSolideci} in $  C_w ( [0,T ]  ;  \mathcal{H}) \cap  L^2 ( 0,T  ; \underline{\mathcal{V}})$.
Moreover this solution satisfies the following energy inequality: for almost any  $t \in [0,T ]$,
\begin{equation}
\label{NSBodyWeakEnergy}
\frac{1}{2} \| u (t, \cdot)  \|^{2}_{\mathcal{H}} + 2 \nu \int_{(0,t) \times \mathcal{F}_0   } | D(u) |^{2}   + 2\alpha\nu \int_0^t \int_{ \partial \mathcal{S}_0} |u- u_\mathcal{S}|^2
\leq \frac{1}{2}   \| u_{0}  \|^{2}_{\mathcal{H}}     .
\end{equation}
\end{Theorem}
Theorem \ref{NSBodyWeak} is the counterpart of  Theorem 4.5 of \cite{Serre} for the Navier conditions instead of the no-slip conditions.

\begin{proof}
We will proceed in several steps.
In particular because the space of test functions $\mathcal{V}$ involves a weight which makes it smaller than the space $\underline{\mathcal{V}}$ involved by the energy estimates, we will first introduce a truncation of the solid velocity far from the solid. This strategy was already used in \cite{Ortega2} in a slightly different context. \par \ \par

\textbf{Truncation.}
Let $R_0 > 0$ be such that $\mathcal{S}_0 \subset B(0, \frac{R_0}{2})$.
For $R > R_0 $, let $\chi_R : \R^3 \rightarrow \R^3 $ be a smooth vector field such that $ r \wedge  \chi_R $ is divergence free, $\chi_R  (x) = x$ for $x \in  B(0,R)$ and satisfying
$\|  \chi_R \|_{L^\infty ( \R^3 ;  \R^3)} \leq R$.
Indeed one may define for example  $\chi_R$ by the formula $\chi_R  (x) =  \frac{R}{| x |} x$ for $x \in  \R^{3} \setminus B(0,R)$.

Observe in particular that for any $r\in \R^{3}$, for any $w \in \mathcal{V}$,
\begin{equation}
\label{cvd}
(r\wedge \chi_R  ) \cdot \nabla w  \rightarrow (r\wedge x) \cdot \nabla w  \text{ in } L^2 ( \R^3 ) , \text{ when } R \rightarrow + \infty ,
\end{equation}
by Lebesgue's dominated convergence theorem.

Then we truncate the solid velocity $ u_{\mathcal{S}}$ defined in \eqref{Svelo} by
$  u_{\mathcal{S},R} (t,x) := \ell  (t)  + r(t) \wedge\chi_R (x)  $,
and we introduce the form
\begin{equation*}
 b_R (u,v,w) := m  \det (r_u ,\ell_{v} ,  \ell_w ) + \det (\mathcal{J}_0 r_u , r_{v} , r_w ) +
\int_{\mathcal{F}_0 } \Big(     [   (u-  u_{\mathcal{S},R} )  \cdot\nabla w ] \cdot v   -  \det (r_u , v ,  w ) \Big).
\end{equation*}
The interest of such a truncation  $b_R$ of $b$ is that
it is now well-defined and trilinear on $ \underline{\mathcal{V}} \times  \underline{\mathcal{V}} \times  \underline{\mathcal{V}}$ (note that the third argument is here  taken not only in $\mathcal{V}$ but in the larger space $ \underline{\mathcal{V}}$) and continuous in the sense that
 there exists a constant $C>0$ such that
for any $(u,v,w) \in  \underline{\mathcal{V}} \times  \underline{\mathcal{V}} \times \underline{\mathcal{V}}$,
\begin{equation}
\label{bcontiR}
|b_{R} (u,v,w)  |  \leqslant C \| u  \|_{\underline{\mathcal{V}}} \,  \| v  \|_{\underline{\mathcal{V}}} \,  \| w \|_{\underline{\mathcal{V}}} .
\end{equation}
Of course the constant $C$ in \eqref{bcontiR} depends on $R$. However, when restricting $b_{R}$ to $ \underline{\mathcal{V}} \times  \underline{\mathcal{V}} \times {\mathcal{V}}$, there exists $C > 0$ such that for any $R > R_0 $, for any $(u,v,w) \in  \underline{\mathcal{V}} \times  \underline{\mathcal{V}} \times {\mathcal{V}}$,
\begin{equation}
\label{bcontiRu}
|b_{R} (u,v,w)  |  \leqslant C \| u  \|_{\underline{\mathcal{V}}} \,  \| v  \|_{\underline{\mathcal{V}}} \,  \| w \|_{\mathcal{V}} .
\end{equation}
We also have that  there exists $C  > 0$ such that for any $R > R_0 $, for any $(u,v) \in  \underline{\mathcal{V}}  \times {\mathcal{V}}$,
\begin{equation}
\label{bcontiRu4}
|b_{R} (u,u,v)  |  \leqslant C (  \| u  \|^{2}_{L^4(\mathcal{F}_{0})}   + \| u  \|^{2}_{\mathcal{H}}  ) \, \| v \|_{\mathcal{V}}   .
\end{equation}
Actually, estimates \eqref{bcontiRu} and \eqref{bcontiRu4} are proved by proceeding in the same way than for the proof of \eqref{bconti}.

 Moreover  the cancellation property \eqref{Tonga} is still correct:
\begin{equation}
\label{TongaR}
(u,v) \in \underline{\mathcal{V}} \times   \underline{\mathcal{V}} \text{ implies } b_R (u,v,v) = 0 .
\end{equation}

Finally we deduce from \eqref{cvd} that for any $(u,v,w) \in \underline{\mathcal{V}} \times  \underline{\mathcal{V}} \times {\mathcal{V}}$,
\begin{equation}
\label{cvdb}
b_{R} (u,v,w)   \rightarrow   b (u,v,w)  \text{ when } R \rightarrow + \infty .
\end{equation}
\par \ \par

\textbf{Existence for the truncated system.} Then, given $u_{0} \in  \mathcal{H}$  and $T >0$, there exists  $u_R $ in $  C_w ( [0,T ]  ;  \mathcal{H}) \cap  L^2 ( 0,T  ; \underline{\mathcal{V}})$ satisfying,  for any $v \in C^{\infty } ([0,T ] ;  \mathcal{H} )$ such that $v |_{\overline{\mathcal{F}_0}} \in C^{\infty } ([0,T ] ;  C^{\infty }_{c} (\overline{\mathcal{F}_0} ))$, and
for all $t\in [0,T ] $,
\begin{equation}
\label{WeakNSR}
(u_R ,v)_{\mathcal{H}}  (t) -  (u_{0},v |_{t=0})_{\mathcal{H}}   =
 \int_{0}^{t} \Big[ (u_R , \partial_{t} v)_{\mathcal{H}}
 + 2  \nu a(u_R ,v) +  b_R (u_R ,u_R ,v)    \Big] .
\end{equation}

Moreover $u_{R}$ verifies
for almost any  $t \in [0,T ]$ the energy inequality
\eqref{NSBodyWeakEnergy}.

This can be proved with very standard methods, by example considering some Faedo-Galerkin approximations and passing to the limit.
Let us therefore sketch a proof of it  referring for example  to \cite{TemamNS} where a comprehensive study of the Leray theorem for the classical case of a fixed boundary is treated.

Let  $(w_j )_{j \geq 1}$ be a Hilbert basis of  $\underline{\mathcal{V}} $.  For simplicity, since the set
\[ \mathcal{Y}:= \{ \phi \in C^\infty_c(\R^{3}) / \ \div \phi = 0  \text{ in }  \R^{3} \text{ and } D(\phi) = 0  \text{ in }  \mathcal{S}_0 \}
 \]
 is dense in $\underline{\mathcal{V}} $,  we take $ w_j \in \mathcal{Y}$, for all $j \geq 1$.

We define an approximate solution $u_{N} := u_{N,R}$ (in the sequel we will omit the dependence on $R$ for the sake of clarity) of the form
$u_{N} = \sum_{i=1}^{N} g_{iN} (t) w_{i}$
satisfying, for any $j=1,\ldots,N$,
\begin{eqnarray}
\label{Ga1}
(\partial_{t }u_{N} ,w_{j})_{\mathcal{H}}    =
 2  \nu a(u_{N} ,w_{j}) +  b_R (u_{N},u_{N},w_{j})    ,
 \\ \label{Ga2}
 u_{N} \vert_{t=0} = u_{N0} ,
\end{eqnarray}
where $ u_{N0}$ is the orthogonal projection in $\mathcal{H}$ of $u_{0}$ onto the space spanned by $w_{1} , \ldots, w_N$.
Let us explain why  such $(u_{N})_{N}$ do exist.
First we introduce the matrices:
\begin{eqnarray*}
\mathcal{M}_{N} := \begin{bmatrix}  ( w_{i} , w_{j} )_{\mathcal{H}} \end{bmatrix}_{1 \leq i,j \leq N} , \quad
\mathcal{G}_{N} := \begin{bmatrix} g_{1N} & \ldots & g_{NN}  \end{bmatrix} , \quad
\mathcal{A}_{N} := \begin{bmatrix}  a( w_{i} , w_{j} ) \end{bmatrix}_{1 \leq i,j \leq N} ,
\end{eqnarray*}
and
for any $u,v \in \R^{N}$, $ \mathcal{B}_{N} (u,v)  := (\mathcal{B}_{Nj} (u,v) )_{1 \leq j \leq N} , \text{ where }  \mathcal{B}_{Nj} (u,v) := \sum_{1 \leq i,k \leq N} u_{i} v_{k} b_R (w_{i} ,w_{k} , w_{j})$.
Then the equation \eqref{Ga1} can be recast as the following nonlinear differential system for the functions $( g_{iN} )_{1\leq i \leq N} $:
\begin{equation*}
\mathcal{G}_{N} '(t) =\mathcal{M}_{N}^{-1} \Big(  2  \nu  \mathcal{A}_{N} \mathcal{G}_{N} + \mathcal{B}_{N} ( \mathcal{G}_{N} , \mathcal{G}_{N})  \Big)
\end{equation*}
and the initial condition \eqref{Ga2} is equivalent to an initial condition of the form
$\mathcal{G}_{N} (0) = \mathcal{G}_{N,0} $.
According to the Cauchy-Lipschitz theorem this system has a maximal solution defined on some time interval $[0,T_{N} ]$ with $ T_{N} > 0$.
Moreover if  $ T_{N}  < T$ then $\|  u_{N} \|_{\mathcal{H}} $ must tend to $+\infty$ as $t \rightarrow T_{N}$.

The following energy estimate shows that this does not happen and therefore $T_{N} = T$.
For any $j=1,\ldots,N$, we multiply \eqref{Ga1} by $ g_{jN} (t)$ and we sum the resulting identities to obtain, thanks to \eqref{TongaR},
$\frac12 \partial_{t }  \|  u_{N} \|^2_{\mathcal{H}}    =
 2  \nu a(u_{N} , u_{N} )    $,
so that, by integration in time, we have
\begin{equation*}
\frac{1}{2} \| u_{N} (t, \cdot)  \|^{2}_{\mathcal{H}} + 2 \nu \int_{(0,t) \times \R^{3} } | D(u_{N}) |^{2}   + 2\alpha \nu\int_0^t \int_{ \partial \mathcal{S}_0} |u_{N}- u_{N,\mathcal{S}}|^2
\leq \frac{1}{2}   \| u_{N0}  \|^{2}_{\mathcal{H}}   \leq \frac{1}{2}   \| u_{0}  \|^{2}_{\mathcal{H}}   .
\end{equation*}
In particular, by the Korn inequality, the sequence $(u_{N}  )_{N}$ is bounded in  $L^{\infty} (0,T ; \mathcal{H}) \cap L^{2} (0,T ;  \underline{\mathcal{V}})$.
Therefore, there exists a subsequence of $(u_{N}  )_{N}$, relabelled the same, converging  weakly-* in $  L^\infty ( 0,T   ;  \mathcal{H}) $ and  weakly in $ L^2 ( 0,T  ;\underline{\mathcal{V}} )$ to $u \in  L^\infty ( 0,T   ;  \mathcal{H}) \cap L^2 ( 0,T  ; \underline{\mathcal{V}})$, as $ N \rightarrow + \infty$, which satisfies,  for almost any $t \in [0,T ]$, the energy inequality \eqref{NSBodyWeakEnergy}.

In order to pass to the limit in the nonlinear term in \eqref{Ga1}, we need a strong convergence. We will closely follow the classical arguments in Ch. 3 of \cite{TemamNS}.
To this end, we are going to bound a fractional derivative in time of the  functions  $u_{N} $ by applying the Fourier transform.
We therefore first extend the functions  $u_{N} $ to the whole time line as follows.
For any $N>1$ let us now denote by $\tilde{u}_{N} $ the function defined from $\R$ to $ \mathcal{H}$ which is equal to $u_{N} $ on $[0,T ]$ and by $0$ outside.
We denote by $\hat{u}_{N} $ the Fourier transform of $\tilde{u}_{N} $, defined by
$\hat{u}_{N}  (\tau) := \int_{\R} e^{-2i\pi t \tau } \tilde{u}_{N}  (t) dt $.
 Similarly we extend the functions $g_{iN}$ by $0$ outside $[0,T ]$ and we denote by $\hat{g}_{iN}$ their respective Fourier transform.

 According to Th. $2.2$ in \cite{TemamNS} it is sufficient to prove that there exists $\gamma >0$ such that
$(| \tau |^{\gamma}  \hat{u}_{N}  (\tau)  )_{N} $ is bounded in $L^{2} (\R ; \mathcal{H})$ to deduce that the sequence  $({u_{N}} )_{N}$ is relatively compact in $L^{2} (0,T ; L^2_{loc} (\mathbb{R}^3)). $

Let us denote, for $t \in \R$, by $\tilde{f}_{N} (t)$ the linear form on $\underline{\mathcal{V}}$ defined by
\begin{align*}
<\tilde{f}_{N} (t), w>   := 2  \nu a(u_{N}(t) ,w) +  b_R (u_{N} (t) ,u_{N} (t),w)    \text{ for } t \in [0,T]   \text{ and }
\tilde{f}_N (t)   := 0 \text{ otherwise,}
\end{align*}
so that the equation \eqref{Ga1} becomes
\[(\partial_{t} \tilde{u}_{N} ,w_{j} )_{\mathcal{H}}    =  - ( u_{N0}  ,w_{j} )_{\mathcal{H}} \,  \delta_{0} (t)  +  (u_{N} \vert_{t=T}  ,w_{j} )_{\mathcal{H}} \,  \delta_{T}(t)   +   <\tilde{f}_{N} , w_{j}>  ,
\]
and then, taking the Fourier transform in time, we get  for any $\tau \in \R$,
\begin{equation*}
2\pi i \tau (\hat{u}_{N}  ,w_{j} )_{\mathcal{H}}    =  - ( u_{N0}  ,w_{j} )_{\mathcal{H}} +   ( u_{N} \vert_{t=T}  ,w_{j} )_{\mathcal{H}} \,  e^{-2i\pi T \tau } + <  \hat{f}_{N}  , w_{j}>  .
\end{equation*}

This yields, multiplying by $\hat{g}_{jN}$ and summing  over $1 \leq j \leq N$, that, for any $\tau \in \R$,
\begin{equation*}
2\pi i \tau   \|   \hat{u}_{N}  (\tau )  \| ^{2}_{\mathcal{H}}    =  - ( u_{N0}  ,  \hat{u}_{N}   )_{\mathcal{H}} +  (  u_{N} \vert_{t=T}  ,  \hat{u}_{N}  )_{\mathcal{H}} \,  e^{-2i\pi T \tau } + <  \hat{f}_{N}  ,  \hat{u}_{N} >  .
\end{equation*}
Thanks to \eqref{Conta} and \eqref{bcontiR}, there exists $C>0$ such that for any $t \in [0,T ]$,
$ \|   f_{N} (t)  \|_{\underline{\mathcal{V}}'}
   \leq C (  \|  u_{N}(t) \|_{\underline{\mathcal{V}}}     +  \| u_{N} (t)  \|^{2}_{\underline{\mathcal{V}}} ) $.
Moreover, for any $ \tau \in \R$,
$ \| \hat{f}_{N} (\tau )   \|_{\underline{\mathcal{V}}'}  \leq \int_0^T   \|  f_{N} (t)  \|_{\underline{\mathcal{V}}'} dt $.
Thus, $ (\sup_{\tau \in \R} \|      \hat{f}_{N}(\tau ) \|_{\underline{\mathcal{V}}'} ) $ is bounded,
and the initial and final values $u_{N0}$ and $ u_{N} \vert_{t=T}$ are bounded as well.
Therefore there exists $C>0$ such that
$\tau   \|   \hat{u}_{N}  (\tau )  \| ^{2}_{\mathcal{H}}   \leq C \|      \hat{u}_{N}(\tau ) \|_{\underline{\mathcal{V}}}   $.
Now, we observe that there exists $C>0$ such that for any $\tau \in \R$,
$| \tau |^{\frac14} \leq C (1+| \tau | ) (1+| \tau |)^{-\frac34} $,
to deduce that
\begin{eqnarray*}
\int_\R  | \tau |^{\frac14} \|   \hat{u}_{N}  (\tau )  \| ^{2}_{\mathcal{H}} d\tau
& \leq & C \int_\R  \frac{1+| \tau | }{1+| \tau |^{\frac34}} \|   \hat{u}_{N}  (\tau )  \|^{2}_{\mathcal{H}} d\tau
 \\ & \leq & C \int_\R  \frac{1 }{1+| \tau |^{\frac34}} \|   \hat{u}_{N}  (\tau )  \|_{\underline{\mathcal{V}}} d\tau
 + C \int_\R \|   \hat{u}_{N}  (\tau )  \|^{2}_{\underline{\mathcal{V}}} d\tau
  \\ & \leq & C  \int_\R \|   \hat{u}_{N}  (\tau )  \|^{2}_{\underline{\mathcal{V}}} d\tau ,
\end{eqnarray*}
by the Cauchy-Schwarz inequality.
Then it follows from the Parseval identity that $(| \tau |^{\frac18}  \hat{u}_{N}  (\tau)  )_{N} $ is bounded in $L^{2} (\R ; \mathcal{H})$.

Then, we can classically pass to the limit in \eqref{Ga1} as $ N \rightarrow \infty$, and obtain that \eqref{WeakNSR} is satisfied.

\par \ \par

\textbf{Endgame.} Since the bounds given by the energy estimate \eqref{NSBodyWeakEnergy} are uniform with respect to $R > R_{0}$, there exists a subsequence
$(u_{R_{k}})_{k}$ converging to $u \in C_w ( [0,T ]  ; \mathcal{H} ) \cap L^2 ( 0,T  ;  \underline{\mathcal{V}}  )$ for the weak (or weak-*) topologies, satisfying \eqref{NSBodyWeakEnergy} for almost any  $t \in [0,T ]$.
This allows to pass to the limit in all the terms involved in
\eqref{WeakNSR} except for the trilinear term.

On the other hand it is sufficient to prove for any $v \in C^\infty ([0,T ] ;  \mathcal{H} )$ such that $v |_{ (0,T)  \times \overline{\mathcal{F}_0}} \in C^{\infty }_{c} ((0,T) \times \overline{\mathcal{F}_0} )$
\[
0 = \int_{0}^{t} \Big[ (u, \partial_{t} v)_{\mathcal{H}}
 + 2  \nu a(u,v) +  b(u,u,v)    \Big] ,
\]
to deduce, by standard arguments, that $ u $ is a weak solution of \eqref{chNS1}-\eqref{chSolideci}.

  It therefore only remains to prove that there exists a subsequence, still denoted
$(u_{R_{k}})_{k}$, such that   for any  $v \in C^\infty ([0,T ] ;  \mathcal{H} )$ with $v |_{ (0,T)  \times \overline{\mathcal{F}_0}} \in C^{\infty }_{c} ((0,T) \times \overline{\mathcal{F}_0} )$,  as $ k \rightarrow \infty$,
\begin{equation}
\label{endg}
 \int_{0}^{t}  b_{R_{k}} (u_{R_{k}} ,u_{R_{k}} ,v)   \rightarrow \int_{0}^{t}  b (u ,u ,v) .
\end{equation}
First let us observe that to prove \eqref{endg}  it will be enough to show that  the sequence $(u_{R_{k}})_{k}$ is relatively compact in $L^2 ( 0,T  ; L^{2}_{loc} (\R^{3}))$.
Indeed this yields that there exists a subsequence, still denoted
$(u_{R_{k}})_{k}$,  converging to $u$ in $L^2 ( 0,T  ; L^{2}_{loc} (\R^{3}) )$, and then
we use the decomposition
\begin{equation*}
 b_{R_{k}} (u_{R_{k}} ,u_{R_{k}} ,v) -  b (u ,u ,v) =
 b_{R_{k}} (u , u ,v) -  b (u ,u ,v)
+  b_{R_{k}} (u_{R_{k}} - u,u_{R_{k}} ,v)
-  b_{R_{k}} (u  , u -u_{R_{k}} ,v) .
\end{equation*}
Observe that we can bound
\[ |b_{R_k}(u,\bar{u},v ) | \leq  C \|u|_K \|_{\mathcal{H}_K} \, \| \bar{u}|_K \|_{\mathcal{H}_K} \, \| v \|_{Lip(K)}, \]
where $C$ is independent of $R_k$, the set $K $ is such that supp $v \subset K $  and $ \|\cdot \|_{\mathcal{H}_K}$ is defined by
\[
\|\phi    \|_{\mathcal{H}_K}^2 :=    \int_{\mathcal{F}_0 \cap K } | \phi   |^2  dx   +  \int_{\mathcal{S}_0 \cap K}  \rho_{S_{0}} | \phi  |^2 dx .
\]
Hence, \eqref{endg} follows from the local strong convergence and \eqref{cvdb}.

Now, with the purpose of proving that $(u_{R_{k}})_{k}$  is relatively compact  in $L^2 ( 0,T  ; L^{2}_{loc} (\R^{3}))$, we are going to establish an a priori estimate of the time derivative of the functions $u_{R_{k}}$. Of course we already have such an estimate thanks to the Fourier transform in time applied above, but this estimate is not uniform in $R$, since we relied on the inequality  \eqref{bcontiR} which is not uniform in $R$.  Instead we are going to prove that 
$(\partial_{t} u_{R_{k}})_{k}$  is bounded in  $L^{\frac{4}{3}} ( 0,T  ; {\mathcal{V}}' )$, relying on the estimate \eqref{bcontiRu4}, which is  uniform in large $R$, rather than on \eqref{bcontiR}.
Then, by using a standard cut-off function, we can apply the Aubin-Lions lemma, see for instance \cite{Simon}[Cor.4], to conclude the desired compactness.

The bound of $(\partial_{t} u_{R_{k}})_{k}$  is obtained as follows.
We first combine  the interpolation inequality \eqref{interpo}
with the energy bounds, to see that $(u_{R_{k}})_{k}$  is bounded in
 $L^{\frac{8}{3}} ( 0,T  ; L^4(\mathcal{F}_{0}))$. %
  Next we use \eqref{bcontiRu4} and Holder's inequality to get that there exists $C  > 0$ such that for any $ k \in \N$,
  for any $v \in L^4 ( 0,T  ; \mathcal{V} )$,
  \begin{equation*}
|  \int_{0}^{t}  b_{R_{k}} (u_{R_{k}} ,u_{R_{k}} , v)  | \leq C \| v   \|_{L^{4}  ( 0,T  ; \mathcal{V})} .
\end{equation*}

 Then we easily infer from \eqref{WeakNSR} the desired estimate of $(\partial_{t} u_{R_{k}})_{k}$, and therefore the proof of Theorem \ref{NSBodyWeak} is complete.

\end{proof}

\subsection{A regularity property} \label{regproperty}

In the present case of the Navier conditions, the dynamics of the body benefits from a remarkable regularity property stated in the proposition below.
We will make use a slight variant of \eqref{bconti}, which involves the space
\begin{equation*}
\widehat{\mathcal{V}}:=   \{  \phi \in   \mathcal{V} / \ \phi\vert_{\mathcal{F}_0} \in \text{Lip} (\overline{\mathcal{F}_0} )    \} , \text{ endowed with the norm  } \|  \phi  \|_{\widehat{\mathcal{V}}} := \|  \phi  \|_\mathcal{V} + \|   \phi  \|_{ \text{Lip} ( \overline{\mathcal{F}_0})   }  .
\end{equation*}
Then one may extend  $b$ to  $\mathcal{H} \times \mathcal{H} \times \widehat{\mathcal{V}}$ such that  there exists a constant $C>0$ such that
for any $(u,v,w) \in \mathcal{H} \times \mathcal{H} \times \widehat{\mathcal{V}}$,
\begin{equation}
\label{bcontiE}
|b(u,v,w)  |  \leqslant C \| u  \|_{\mathcal{H}} \,  \| v  \|_{\mathcal{H}} \,  \| w \|_{\widehat{\mathcal{V}}} .
\end{equation}
Let us emphasize for the comfort of the reader that $\widehat{\mathcal{V}} \subset \mathcal{V} \subset \underline{\mathcal{V}}$.

Let us also denote by $\lambda_i, \, i =1,2,3$ the eigenvalues of the inertial matrix $ \mathcal{J}_0$, which is symmetric definite positive, so that, $ \lambda_i > 0 $ for all $i=1,2,3 $. Moreover, we consider the spectral norms $ \|\mathcal{J}_0 \| := \max (\lambda_i)$ and $ \|\mathcal{J}_0^{-1}\|^{-1} := \min (\lambda_i)$.

\begin{Proposition}
\label{added}
Let be given $u_{0} \in  \mathcal{H}$  and $T >0$.
Consider  a weak solution $u$ of \eqref{chNS1}-\eqref{chSolideci}  given by Theorem \ref{NSBodyWeak}.
Then $\ell $ and $r$ are in $H^{1}  ( 0,T  ; \R^{3})$ and satisfy the following: there exist
\begin{enumerate}
\item a  $6 \times 6$  definite positive symmetric  matrix $\mathcal{M}$ depending only on
$\mathcal{S}_0$, $m$ and $ \mathcal{J}_0$ such that
there exist $ \underline{m} >0 $ and $ \beta > 0 $ depending only on $ \mathcal{S}_0$ such that, for any  $F$ and $T$ in $ \mathbb{R}^3 $,
\begin{equation} \label{estimateM}
\Bigl\|\mathcal{M}^{-1} \begin{bmatrix}
F \\ T \end{bmatrix}  \Bigr\|  \leq 2 ( m^{-1} \|F \| + \|\mathcal{J}_0^{-1}\| \|T \|) \ \text{ for } m \geq \underline{m}, \, \text{ and } \,  \lambda_i \geq \beta,\, i =1,2,3 ,
\end{equation}
\item some functions $(v_i )_{i \in \{1,\ldots,6\}}$ in $  \widehat{\mathcal{V}}$  depending only on
$\mathcal{S}_0$,
\end{enumerate}
such that there holds in $L^2 ( 0,T)$:
\begin{equation}
\label{AM2}
    \mathcal{M}  \begin{bmatrix} \ell \\ r \end{bmatrix} '
   = ( 2\nu a(u, v_i )  + b(u,u, v_i) )_{i \in \{1,\ldots,6\}}  .
\end{equation}
\end{Proposition}
\begin{proof}
The matrix $\mathcal{M}$ is usually referred to as virtual inertia tensor, it incorporates the added mass of the solid $\mathcal{M}_2$, and is defined by
\begin{equation*}
  \mathcal{M}_{1}
:= \begin{bmatrix} m \Id_3 & 0 \\ 0 & \mathcal{J}_0 \end{bmatrix}  ,   \quad
  \mathcal{M}_{2} := \begin{bmatrix} \displaystyle\int_{\mathcal{F}_0} \nabla \Phi_i \cdot \nabla \Phi_j \ dx \end{bmatrix}_{i,j \in \{1,\ldots,6\}}  \text{ and } \mathcal{M} := \mathcal{M}_{1} + \mathcal{M}_{2} ,
\end{equation*}
where the functions $\Phi_{i}$, usually referred to as the Kirchhoff potentials,  as the solutions of the following problems:
\begin{equation*}
	-\Delta \Phi_i = 0 \quad   \text{for}  \ x\in \mathcal{F}_{0} ,
\end{equation*}
\begin{equation*}
	 \Phi_i \rightarrow 0 \quad  \text{as}  \ |x| \rightarrow  \infty,
\end{equation*}
\begin{equation*}
	\frac{\partial \Phi_i}{\partial n}=K_i
		  \quad  \text{for}  \  x\in \partial \mathcal{S}_{0},
\end{equation*}
where
\begin{equation*}
K_i:= \left\{\begin{array}{ll}
n_i & \text{if} \ i=1,2,3 ,\\ \relax
[ x\wedge n]_{i-3}& \text{if} \ i=4,5,6.
\end{array}\right.
\end{equation*}
We observe that the matrix $\mathcal{M}_{2}$ depends only on $ \mathcal{S}_{0}$ and is nonnegative symmetric
so that $\mathcal{M}$ depends only on
$\mathcal{S}_0$, $m$ and $ \mathcal{J}_0$, and  is definite positive symmetric.

Now for any  $F$ and $T$ in $ \mathbb{R}^3 $, let
\[   \begin{bmatrix}
x \\ y \end{bmatrix}   := \mathcal{M}^{-1} \begin{bmatrix}
F \\ T \end{bmatrix}  =   \begin{bmatrix}
m^{-1} \Id_3 & 0  \\ 0 & \mathcal{J}_0^{-1} \end{bmatrix}  \left( \begin{bmatrix}
F \\ T \end{bmatrix}   - \mathcal{M}_2 \begin{bmatrix}
x \\ y \end{bmatrix} \right).\]
Since $ \mathcal{M}_2 $ depends only on $ \mathcal{S}_0 $, there exists $C > 0 $  depending only on $ \mathcal{S}_0$ such that
$\Bigl\| \mathcal{M}_2  \begin{bmatrix}
x \\ y \end{bmatrix}  \Bigl\| \leq C \Bigl\|\begin{bmatrix}
x \\ y \end{bmatrix}  \Bigr\|.$
Then, we can estimate
\[ \Bigl\| \begin{bmatrix}
x \\ y \end{bmatrix}  \Bigr\|\leq m^{-1} \|F \| + \|\mathcal{J}_0^{-1}\|\|T\| + \max\{m^{-1}, \|\mathcal{J}_0^{-1}\|\} C \Bigl\| \begin{bmatrix}
x \\ y \end{bmatrix}  \Bigr\|.\]
It is therefore sufficient to take $ \underline{m} = \beta = 2 C$ to obtain \eqref{estimateM}.

For $i = 1,\ldots,6$, we introduce the function $v_i $ defined by
\begin{equation*}
v_i := \nabla  \Phi_i \text{ in } \mathcal{F}_{0}  \text{ and }v_i :=
\left\{\begin{array}{ll}
e_i & \text{if} \ i=1,2,3 ,\\ \relax
 e_{i-3} \wedge x & \text{if} \ i=4,5,6,
\end{array}\right.
  \text{ in } \mathcal{S}_{0}.
\end{equation*}
These functions only depend on $ \mathcal{S}_{0}$.
Moreover they are in $\widehat{\mathcal{V}}$. Observe in particular that $\nabla v_i $ decays like $1/ | \cdot |^{4}$ at infinity so that $\int_{ \mathcal{F}_0 } | \nabla v_i (y) |^2 (1 + | y |^2 ) dy < + \infty  $, see for instance  \cite{Childress}[4.3.1].
We can therefore take them as test functions in  \eqref{WeakNS}.
Indeed we apply
\eqref{WeakNS} to $v = v_i$ and we derive in time to obtain, for all $t\in [0,T ] $,
\begin{equation}
\label{WeakNS2}
\partial_t (u,v_i)_{\mathcal{H}}   =
  2  \nu a(u,v_i) +  b(u,u,v_i)    .
\end{equation}
Let us  prove that
\begin{equation}
\label{energyajout}
\begin{bmatrix} (u,v_i)_{\mathcal{H}}  \end{bmatrix}_{i,j \in \{1,\ldots,6\}}
=  \mathcal{M}  \begin{bmatrix} \ell \\ r \end{bmatrix} .
\end{equation}
To this end, we first use \eqref{plusimple} to arrive at
\begin{equation}
\label{energyajout1}
\begin{bmatrix} (u,v_i)_{\mathcal{H}}  \end{bmatrix}_{i,j \in \{1,\ldots,6\}}
=
\begin{bmatrix}  \int_{\mathcal{F}_{0} } u \cdot \nabla  \Phi_i  \end{bmatrix}_{i,j \in \{1,\ldots,6\}}
+
\mathcal{M}_1  \begin{bmatrix} \ell \\ r \end{bmatrix} .
\end{equation}
Then, using an integration by parts, we observe that
\begin{equation*}
\int_{\mathcal{F}_{0} } u \cdot \nabla  \Phi_i  = \int_{\partial\mathcal{S}_{0} } (u \cdot n)  \Phi_i   = \int_{\partial\mathcal{S}_{0} } (u_{\mathcal{S}} \cdot n)  \Phi_i  ,
\end{equation*}
so that, expanding $u_{\mathcal{S}}$ and using another integration by parts, give us
\begin{equation}
\label{energyajout2}
\begin{bmatrix}  \int_{\mathcal{F}_{0} } u \cdot \nabla  \Phi_i  \end{bmatrix}_{i,j \in \{1,\ldots,6\}}
=
 \mathcal{M}_2  \begin{bmatrix} \ell \\ r \end{bmatrix} .
\end{equation}
Gathering \eqref{energyajout1} and \eqref{energyajout2} yields \eqref{energyajout}.
Then combining  \eqref{WeakNS2} and \eqref{energyajout} furnishes  \eqref{AM2}.

Therefore it only remains to prove that $\ell $ and $r$ are in $H^{1}  ( 0,T  ; \R^{3})$.
Since the  matrix $\mathcal{M}$ is time-independent it is sufficient to prove that the right hand side of \eqref{AM2} is in $L^{2}  ( 0,T  ; \R^{3})$.
Indeed this follows from \eqref{Conta},  \eqref{bcontiE}, from that $(v_i )_{i \in \{1,\ldots,6\}}$ are in $ \widehat{\mathcal{V}}$ and that $u$ is in $  C_w ( [0,T ]  ;  \mathcal{H}) \cap  L^2 ( 0,T  ; \underline{\mathcal{V}})$.
\end{proof}

Let us emphasize that this property seems to be a particularity of the case of the Navier conditions. In particular in the case of the no-slip conditions, the corresponding weak formulation
 involves test functions continuous across the body's boundary, a feature which is not satisfied by
the functions $v_i $.
To our knowledge the counterpart of Proposition \ref{added} in the case of the no-slip conditions is not known.

\par
\  \par
The estimate
 \eqref{estimateM}  will be useful for the next section.
It is also interesting for the sequel to observe that
\begin{equation}
\label{ouf}
( b(u,u, v_i) )_{i \in \{1,\ldots,6\}}
=
 \begin{bmatrix} m \, r \wedge \ell  \\ ( \mathcal{J}_0 r ) \wedge r \end{bmatrix}
+ (\int_{\mathcal{F}_0 } \Big(     [    \big( u -  u_\mathcal{S} )  \cdot\nabla   \big) \nabla \Phi_i] \cdot u   -  \det (r_u , u ,   \nabla \Phi_i ) \Big) )_{i \in \{1,\ldots,6\}}  .
\end{equation}

\subsection{The infinite inertia limit}
\label{SectionInertiaNS}

Let us also mention that Theorem \ref{NSBodyWeak}  extends to the case of a moving body some earlier results, in particular see \cite{cmr,Iftimie-Sueur}, about the existence of Leray solutions in the case where  Navier conditions  are considered but on a fixed boundary.
In this case, the system reads
\begin{gather}
\label{chNS1Mass}
\frac{\partial u}{\partial t}+   u\cdot \nabla u   + \nabla p =  \nu \Delta u \ \text{ for } \ x \in \mathcal{F}_0 , \\
\label{chNS2Mass}
\div u = 0 \ \text{ for } \  x \in \mathcal{F}_0  ,  \\
\label{chNS3Mass}
u  \cdot n =  0  \ \text{ for } \ x\in \partial \mathcal{S}_0,  \\
\label{chNS4Mass}
(D(u) n )  \wedge n = - \alpha u \wedge n   \ \text{ for } \ x\in \partial \mathcal{S}_0 ,  \\
\label{chNSci2Mass}
u |_{t= 0} = u_0 ,
\end{gather}
and a weak Leray solution of \eqref{chNS1Mass}-\eqref{chNSci2Mass} is by definition a function
\begin{equation*}
u \in C_w ( [0,T ]  ; L^{2}_{\sigma} (\mathcal{F}_0  ) ) \cap L^2 ( 0,T  ;  H^1 ( \mathcal{F}_0  ) ) ,
\end{equation*}
such that
\begin{enumerate}
\item for all $v \in   H^1 ( 0,T   ; L^{2}_{\sigma} (\mathcal{F}_0  )) \cap L^4 ( 0,T  ;   H^1 ( \mathcal{F}_0  ) )$,
and for all $t\in [0,T ] $,
\begin{equation}
\label{WeakNSMass}
  \int_{\mathcal{F}_0 }  u  (t,\cdot) \cdot   v  (t,\cdot) dx - \int_{\mathcal{F}_0 }  u_{0}   \cdot   v \vert_{t=0} dx
  =
  \int_{0}^{t} \Big[
   \int_{\mathcal{F}_0 }  u   \cdot \partial_{t}   v dx
  + 2  \nu a^{*} (u ,v) +  b^{*} (u ,u ,v)    \Big] ,
\end{equation}
where
\begin{eqnarray*}
a^{*} (u,v) := - \alpha  \int_{\partial \mathcal{S}_0} u  \cdot v  - \int_{\mathcal{F}_0 }     D(u) :  D(v)
,\quad \text{and} \quad   b^{*} (u,v,w) := \int_{\mathcal{F}_0 }  [   u  \cdot\nabla w ] \cdot v ,
\end{eqnarray*}
\item for any $t \in [0,T ]$,
\begin{equation*}
\frac{1}{2} \| u (t, \cdot)  \|^{2}_{L^{2} (\mathcal{F}_0  )} + 2 \nu \int_{(0,t) \times  \mathcal{F}_0} | D(u) |^{2}   + 2\alpha\nu \int_0^t \int_{ \partial \mathcal{S}_0} |u|^2
\leq \frac{1}{2}   \| u_{0}  \|^{2}_{L^{2} (\mathcal{F}_0 ) }     .
\end{equation*}
\end{enumerate}

Here, $L^{2}_{\sigma}  (\mathcal{F}_0  )$  denotes the space of the divergence free vector fields in $L^{2}  (\mathcal{F}_0  )$ which are tangent to the solid's boundary $\partial \mathcal{S}_0 $.
\par \ \par
The following result shows that the case with the body fixed can be thought as \textit{the limit of infinite inertia}, that is, when $m$ and the eigenvalues $(\lambda_i )_{i =1,2,3} $ of $\mathcal{J}_0 $ converge to $+\infty$ with $\lambda_i = O(\lambda_j )$ for any $i,j$. Let us observe that the eigenvalues of $\mathcal{J}_0$ are required to diverge at the same order.
  This last condition is quite natural if one thinks that  the solid density $\rho_{{\mathcal S}_0}$ is multiplied by a factor converging to $+\infty$
in \eqref{EqMasse} and \eqref{eqJ}.
This condition can alternatively be written as   $ \|\mathcal{J}_0 \| = O( \|\mathcal{J}_0^{-1}\|^{-1} )$ (using the spectral norms introduced in Section   \ref{regproperty}).

Let us point out that $\| \mathcal{M}^{-1}  \|_{\R^{6 \times 6}} \rightarrow 0$ as $m$ and $ (\lambda_i)_{i =1,2,3},$ converge to $+\infty$, as a consequence of
the estimate   \eqref{estimateM}, and therefore in particular
$\| \mathcal{M}^{-1}  \|_{\R^{6 \times 6}} \rightarrow 0$  in the infinite inertia limit.
 Another observation that will be useful is that there holds for any $r \in \R^{3}$,
\begin{equation}
\label{cdotVSwedge}
\|  ( \mathcal{J}_0 r ) \wedge r \| \leq C  ( \mathcal{J}_0 r ) \cdot r  ,
\end{equation}
for a constant $C >0$, uniform in the infinite inertia limit.

Indeed, introducing some normalized eigenvectors $ (r_{i})_{i =1,2,3},$  associated to  $ (\lambda_i)_{i =1,2,3},$ respectively, we can write for some real coefficients $ (\alpha_i)_{i =1,2,3},$ that
$r = \sum_{i=1}^{3} \alpha_{i} r_{i} $ and therefore
$( \mathcal{J}_0 r ) \wedge r = \sum_{i,j} \alpha_i \alpha_j \lambda_i (r_i \wedge r_{j} ) $.
So that, for some constants uniform in the infinite inertia limit, one has:
\begin{equation*}
\|  ( \mathcal{J}_0 r ) \wedge r \| \leq  C
 \sum_{i,j} |  \alpha_i  |     | \alpha_j  |   \lambda_i
 \leq  C'   \sum_{i,j} |  \alpha_i  |     | \alpha_j  |  \sqrt{ \lambda_i } \sqrt{ \lambda_j }
 \leq C'' (  \sum_i |  \alpha_i  |    \sqrt{ \lambda_i } )^{2}
 \leq  C'''  \sum_i  \alpha_i^{2}  \lambda_i = C'''  ( \mathcal{J}_0 r ) \cdot r   .
\end{equation*}
\begin{Theorem}
\label{NSBodyWeakMassless}
Let be given $u_{0} \in  \mathcal{H}$  with $\ell_{0} = r_{0} = 0$ and $T >0$.
For any $m$ and $ \mathcal{J}_0 $ we
consider  a weak solution $u$ of \eqref{chNS1}-\eqref{chSolideci} in $C_w ( [0,T ]  ;  \mathcal{H}) \cap  L^2 ( 0,T  ; \underline{\mathcal{V}}  )$ given by Theorem \ref{NSBodyWeak}.
Then  in the infinite inertia limit, $u \vert_{\mathcal{F}_0}$ converges, up to a subsequence,  in $L^2 ( 0,T; L^{2}_{loc} (\mathcal{F}_0  )   )$ to  a weak solution of \eqref{chNS1Mass}-\eqref{chNSci2Mass} and $\ell $ and $r$ converge to $0$ in $H^{1}  ( 0,T  ; \R^{3})$.
\end{Theorem}
\begin{proof}
We infer from \eqref{NSBodyWeakEnergy} that  $u$  is bounded in $L^{\infty} (0,T ; \mathcal{H}) \cap L^{2} (0,T ;  \underline{\mathcal{V}}  )$ uniformly with respect to the inertia.
Not only that but we also obtain that $ \ell $ and $r$ converge to $0$ in $L^{\infty}  ( 0,T;\R^3 )$  in the infinite inertia limit, because $ (\mathcal{J}_0 r )\cdot r \geq \min(\lambda_i) \|r\|^2.$

The following lemma is quite simple to establish but will be useful in the sequel.
\begin{Lemma}
\label{Lemmapeupres}
Let $u$ be as in Theorem \ref{NSBodyWeakMassless}. Then for any   $v \in   H^1 ( 0,T   ; L^{2}_{\sigma} (\mathcal{F}_0  )) \cap L^4 ( 0,T  ;   \tilde{\mathcal{V}}  )$, for any $t\in [0,T ] $,
\begin{equation}
\label{peupres}
  \int_{\mathcal{F}_0 }  u  (t,\cdot) \cdot   v  (t,\cdot) dx - \int_{\mathcal{F}_0 }  u_{0}   \cdot   v \vert_{t=0} dx
  =
  \int_{0}^{t} \Big[
   \int_{\mathcal{F}_0 }  u   \cdot \partial_{t}   v dx
  + 2  \nu a^{*} (u ,v) +  b^{*} (u ,u ,v)  + F(u ,v)    \Big] ,
\end{equation}
where
\begin{equation*}
F(u ,v)  :=  2 \alpha \nu \int_{\partial \mathcal{S}_0}  u_\mathcal{S}  \cdot v
- \int_{\mathcal{F}_0 } \Big(     [    u_\mathcal{S}   \cdot\nabla v ] \cdot u   -  \det (r_u , u ,  v ) \Big) .
\end{equation*}
\end{Lemma}
Above $\tilde{\mathcal{V}}$  denotes the space
\begin{equation*}
\tilde{\mathcal{V}} := \{  \phi \in  L^{2}_{\sigma}  (\mathcal{F}_0  ) / \ \int_{ \mathcal{F}_0 } | \nabla \phi  (y) |^2 (1 + | y |^2 ) dy < + \infty   \} ,
\end{equation*}
endowed with the norm
\begin{equation*}
\|   \phi \|_{\tilde{\mathcal{V}}} := \|  \phi    \|_{ L^{2}  (\mathcal{F}_0  )} + \|  \nabla \phi  \|_{L^{2}  (\mathcal{F}_0  , (1 + | y |^2 )^{\frac12} dy)} .
\end{equation*}
\begin{proof}[Proof of Lemma \ref{Lemmapeupres}]
It is sufficient to extend $v$  by $0$ in $\mathcal{S}_0$ to obtain a function in $  H^1 ( 0,T   ; \mathcal{H} ) \cap L^4 ( 0,T  ;  \mathcal{V} )$ which is used as a test function in \eqref{WeakNS}.
 This provides \eqref{peupres}.
\end{proof}

Then,  proceeding as in Section \ref{Leray}, we obtain a bound of  $\partial_{t} u \vert_{\mathcal{F}_0}$ in  $L^{\frac{4}{3}} ( 0,T  ; \tilde{\mathcal{V}}' )$  which is uniform in the infinite inertia limit.
We therefore deduce that the sequence of weak solutions  $u $ is relatively compact in $L^{2} (0,T ; L^2_{loc} (  \mathcal{F}_0   ))$.
Thus the restrictions of $u$ to $\mathcal{F}_0 $ converge, up to a subsequence,  to a limit $u^*$ weakly-* in $L^{\infty} (0,T ; L^{2} (\mathcal{F}_0  ))$, weakly in $L^2 ( 0,T  ;   H^1 ( \mathcal{F}_0  ) )$ and strongly in  $L^{2} (0,T ; L^2_{loc} (  \mathcal{F}_0 ) ) $.

Let us now prove that $u^*$ is a weak solution of \eqref{chNS1Mass}-\eqref{chNSci2Mass}.
We deduce from the above convergence that $F(u ,v)$ converges to $0$ in $L^1 ( 0,T )$, and passing to limit in the other terms of \eqref{peupres}
we can conclude that $u^*$ satisfies \eqref{WeakNSMass} for any
$v \in   H^1 ( 0,T   ; L^{2}_{\sigma} (\mathcal{F}_0  )) \cap L^4 ( 0,T  ;   H^1 ( \mathcal{F}_0  ) )$ such that
$\int_{ \mathcal{F}_0 } | \nabla v  (y) |^2 (1 + | y |^2 ) dy < + \infty  $.
Then one easily remove this last condition by using  Lebesgue's dominated convergence theorem.
Thus  $u^*$ is a weak solution of \eqref{chNS1Mass}-\eqref{chNSci2Mass}.

In order to finish the proof of Theorem \eqref{NSBodyWeakMassless} it only remains to prove that
$\ell' $ and $r'$ converge to $0$ in $L^{2}  ( 0,T  ; \R^{3})$.
This relies on the regularity property established in the previous section.
Indeed we define, for $i \in \{1,\ldots,6\}$,
\begin{equation*}
\mathcal{T}_{1,i}  := \int_{\mathcal{F}_0 } \Big(     [    \big( u -  u_\mathcal{S} )  \cdot\nabla   \big) \nabla \Phi_i] \cdot u   -  \det (r_u , u ,   \nabla \Phi_i ) \Big) ,
\end{equation*}
and $\mathcal{T}_{2} :=
 \begin{bmatrix} m \, r \wedge \ell  \\ ( \mathcal{J}_0 r ) \wedge r \end{bmatrix} $,
so that from \eqref{AM2} and \eqref{ouf} we infer that
\begin{equation}
\label{AM2new}
   \begin{bmatrix} \ell \\ r \end{bmatrix} '
   =  2\nu \mathcal{M}^{-1} (a (u, v_i )  )_{i \in \{1,\ldots,6\}} + \mathcal{M}^{-1} (  \mathcal{T}_{1,i})_{i \in \{1,\ldots,6\}}   + \mathcal{M}^{-1} \mathcal{T}_{2}  .
\end{equation}

Since $u$  is bounded in $L^{\infty} (0,T ; \mathcal{H}) \cap L^{2} (0,T ;  \underline{\mathcal{V}}  )$, the functions $(v_i )_{i \in \{1,\ldots,6\}}$ are in $  \widehat{\mathcal{V}}$  and depend  only on
$\mathcal{S}_0$,  and
$\| \mathcal{M}^{-1}  \|_{\R^{6 \times 6}} \rightarrow 0$ in the infinite inertia limit, we infer easily from  \eqref{Conta} that the first term of the right hand side of \eqref{AM2new} vanishes in $L^{2}  ( 0,T  ; \R^{6})$
 in the infinite inertia limit.

On the other hand we can bound the second term as follows: for any $i \in \{1,\ldots,6\}$,  for any $t$,
$\| \mathcal{T}_{1,i}  \|   \leq C  ( \|  u \|^{2}_{L^{2} (\mathcal{F}_0  )} + \|  \ell \|^2  + \|r\|^{2} )$.
Hence, thanks to the energy bound,  we get that $\mathcal{T}_{1,i} $ is bounded  in the infinite inertia limit.
Therefore the second term of the right hand side of \eqref{AM2new} also vanishes in $L^{2}  ( 0,T  ; \R^{6})$
 in the infinite inertia limit.

Finally, in order to deal with the last term, we use the  estimate \eqref{estimateM} to bound, for any $t$,
\begin{align*}
\| \mathcal{M}^{-1} \mathcal{T}_{2}  \| &  \leq 2 \Bigl( \| r \wedge \ell  \| + \bigl(\min (\lambda_{i} )\bigr)^{-1} \| (\mathcal{J}_0 r) \wedge r  \| \ \Bigr) \\
& \leq C \Bigl( \|  \ell \|^2  + \|r\|^{2} + \bigl(\min (\lambda_{i} )\bigr)^{-1}  (\mathcal{J}_0 r) \cdot r  \Bigr ) ,
\end{align*}
thanks to \eqref{cdotVSwedge}. We thus deduce from the energy bound that the last term of the right hand side of \eqref{AM2new} also vanishes in $L^{2}  ( 0,T  ; \R^{6})$
 in the infinite inertia limit.

The proof of Theorem \ref{NSBodyWeakMassless} is then complete.
\end{proof}

\section{Smooth local-in-time solutions of the inviscid system}
\label{Sl}

In this section we consider the system ``inviscid incompressible fluid + rigid body''.

\subsection{The system ``inviscid incompressible fluid + rigid body'' }

When the viscosity coefficient $\nu$ is set equal to $0$, formally, the system \eqref{NS1}-\eqref{Solideci} degenerates into the following equations:
\begin{gather}
\label{Euler1}
\frac{\partial U^E}{\partial t} + (U^E \cdot\nabla) U^E  + \nabla P^E  =   0   \ \text{ for } \ x \in \mathcal{F}^E  (t) = \mathbb{R}^3 \setminus
\mathcal{S}^E(t), \\
\label{Euler2}
\div U^E  = 0 \ \text{ for } \  x \in \mathcal{F}^E  (t) ,  \\
\label{Euler3}
U^E  \cdot n =  U^E_{{\mathcal S}}   \cdot n  \ \text{ for } \ x\in \partial \mathcal{S}^E   (t),  \\
\label{EulerSolide1}
m   (h^E) ''  =   \int_{ \partial \mathcal{S}^E  (t)}  P^E  n \, ds ,  \\
\label{EulerSolide2}
(\mathcal{J}^E R^E   )' =    \int_{ \partial   \mathcal{S}^E (t)}  P^E  (x-  h^E(t) ) \wedge n \, ds    , \\
\label{Eulerci2}
U^E  |_{t= 0} = U^E_0 , \\
\label{EulerSolideci}
h^E  (0)=  0 , \ (h^E )' (0)= \ell^{E}_0 ,\ R^{E}  (0)=  r^{E}_0 ,
\end{gather}
where the solid velocity is given by
\begin{equation*}
U^E_{{\mathcal S}}(t,x) := (h^E) '(t)  +  R^E (t)  \wedge  (x-h^E(t)) ,
\end{equation*}
and
\begin{equation*}
\mathcal{S}^E   (t) := \eta^E (t,\cdot) (\mathcal{S}_0 ) ,\text{ with } \eta^E (t,x) := h^E (t) + Q^E (t) x ,
\end{equation*}
where the matrix $ Q^E$ solves the  differential equation
\begin{equation*}
 (Q^E)' x   = R^E  \wedge (Q^E x) \ \text{ with } \
 Q^E (0)x = x , \, \text{ for any } x \in \mathbb{R}^3.
 \end{equation*}
  Finally $ \mathcal{J}^E$ is given by
\begin{equation*}
 \mathcal{J}^E = Q^E \mathcal{J}_0 (Q^E)^{T} .
 \end{equation*}

Observe that we prescribe $h^E  (0)=  0$ so that the initial position $ \mathcal{S}^E (0)$ occupied by the solid also starts from $ \mathcal{S}_0$ at $t=0$.
The mass $m$ and the initial inertial matrix $ \mathcal{J}_0 $ are also the same than in the previous case of the Navier-Stokes equations.
\par
\  \par

Let us emphasize that in the boundary condition \eqref{Euler3} there  is only an impermeability condition, the slip-with-friction condition is no more prescribed.
This loss of boundary condition generates a boundary layer which makes difficult the issue of the inviscid limit of the system since the fluid flow is drastically modified in a neighborhood of the body's boundary of thickness proportional to $\sqrt{\nu}$.
 The main goal of the paper is precisely to show that despite these layers  the solution  of \eqref{NS1}-\eqref{Solideci} converges in a rather good manner to the solution  of \eqref{Euler1}-\eqref{EulerSolideci} as $\nu \rightarrow  0$.
This will be achieved in the next section. Here we will first gather a few results about the inviscid system \eqref{Euler1}-\eqref{EulerSolideci}.

\subsection{A change of variables}

To write the system  in a fixed domain, we perform the following change of coordinates:
\begin{eqnarray*}
\ell^E (t) :=   Q^E(t)^T \,   (h^E )' (t)  ,  \, R^E(t) := Q^E(t)r^E(t) ,  \,
\\ u^E (t,x) := Q^E (t)^T \,  U^E(t, Q^E(t) x+h^E(t))
\text{ and }  p^E(t,x) := P^E(t, Q^E(t) x+ h^E(t)) ,
\end{eqnarray*}
where $Q^E(t)$ is   the rotation matrix associated to the motion of ${\mathcal{S}}^E (t)$ defined in the previous section.

Observe that this change of variable is analogous to the one that we have used for the Navier-Stokes equations in Section \ref{changeNS}.

The system  \eqref{Euler1}-\eqref{EulerSolideci}   now reads
\begin{gather}
\label{chEuler1}
\frac{\partial u^E }{\partial t} +  ( u^E  - u^E_\mathcal{S} ) \cdot\nabla  u^E  +  r^E  \wedge u^E  + \nabla p^E  =  0 \ \text{ for } \ x \in \mathcal{F}_0  , \\
\label{chEuler2}
\div u^E  = 0 \ \text{ for } \  x \in \mathcal{F}_0  ,  \\
\label{chEuler3}\
u^E  \cdot n = u^E_\mathcal{S}   \cdot n  \ \text{ for } \ x\in \partial \mathcal{S}_0  ,  \\
\label{chEulerSolide1}
m   (\ell^E)' =    \int_{ \partial \mathcal{S}_0 } p^E  n \, ds + (m\ell^E ) \wedge  r^E ,  \\
\label{chEulerSolide2}
\mathcal{J}_0   (r^E)'   =   \int_{ \partial   \mathcal{S}_0} p^E x \wedge
n \, ds  +    ( \mathcal{J}_0  r^E) \wedge r^E  , \\
\label{chEulerci2}
u^E  |_{t= 0} = u^E _0 , \\
\label{chEulerSolideci}
h^E  (0)= 0 , \ (h^E) ' (0)= \ell^E_0 , \ r^E  (0)=  r^E_0 ,
\end{gather}
with
\begin{equation}
\label{LastReadingE}
  u^E_{\mathcal{S}} (t,x) := \ell^E  (t)  + r^E (t) \wedge  x.
   \end{equation}

\subsection{Smooth local-in-time solutions}

Let us recall the following result  from \cite{KatoBody}
about the existence and uniqueness of classical solutions to the equations \eqref{chEuler1}-\eqref{chEulerSolideci}, where, as previously, we extend the initial data  $u^E_{0}$ by setting
 $u^E_{0} := \ell^E_{0} + r^E_{0} \wedge x $ for $x \in \mathcal{S}_0 $.
\begin{Theorem}
\label{EulerBodyStrong}
Let be given  $\lambda \in (0,1)$ and $u^E_{0} \in  \mathcal{H}$ such that $ u^E_{0} |_{{\mathcal{F}_0}} \in  H^{1} \cap C^{ 1,\lambda} $ and $\curl u^E_{0}|_{{\mathcal{F}_0}} $ is compactly supported.
Then there exist $T >0$ and a unique solution $u^E $ of \eqref{chEuler1}-\eqref{chEulerSolideci} in $  C^{1} ( [0,T ]  ;  \mathcal{H})$ such that
$(\nabla u^E ) |_{ [0,T ]  \times {\mathcal{F}_0}} \in    C ( [0,T ]  ;  L^2 ({\mathcal{F}_0} , (1+ | x |^2 )^\frac12  dx )) \cap  C_{w*} ( [0,T ]  ; C^{0,\lambda} ({\mathcal{F}_0}))$.
Moreover for any $t \in [0,T ]$,
\begin{equation}
\label{EulerBodyStrongEnergy}
 \| u^E  (t, \cdot)  \|^2_{\mathcal{H}} =  \|  u^E_{0}  \|^2_{\mathcal{H}}  .
\end{equation}
\end{Theorem}
Here, $C^{k,\lambda} ({\mathcal{F}_0}), \ k \in \mathbb{N}, \lambda \in (0,1)$, denotes the usual H\"older space.
\\

A few comments are in order.

 It is worth to point out here that the solution $u^E $ given by Theorem
\ref{EulerBodyStrong} satisfies the following property: for any $t \in [0,T ]$, for any  $v \in \underline{\mathcal{V}}$,  there holds
\begin{equation}
\label{chEulerweak}
  (\partial_{t}  u^E ,  v)_{\mathcal{H}}  =  - b(u^E,v,u^E )    .
\end{equation}
To see that,  multiply  \eqref{chEuler1} by $v $ and integrate by parts in space using \eqref{chEuler2}-\eqref{chEuler3}.

It is useful for the sequel to recall that the proof given in  \cite{KatoBody} of Theorem
\ref{EulerBodyStrong} relies on the following vorticity reformulation of  \eqref{chEuler1}-\eqref{chEulerSolideci}:
\begin{align}
\label{ExiOmega}
& \partial_{t} \omega^E +  (u^E - u^E_\mathcal{S} ) \cdot \nabla \omega^E =  ( \omega^E \cdot \nabla) (u^E - u^E_\mathcal{S} )\ \text{ in } [0,T] \times  {\mathcal F}_0 ,
\\ \label{Eq:vtransporte}
& \left\{ \begin{array}{l}
\curl {u}^E = {\omega}^E\ \text{ in } [0,T] \times {\mathcal F}_0 , \\
\div {u}^E = 0 \ \text{ in } [0,T] \times {\mathcal F}_0 , \\
{u}^E \cdot n = u^E_\mathcal{S} \cdot n \ \text{ on } [0,T] \times \partial {\mathcal S}_0 ,  \\
 u^E \rightarrow 0 \quad  \text{for}  \ |x| \rightarrow  \infty,
\end{array} \right.
\\&  u^E_\mathcal{S} (t,x) := \ell^E (t) + r^E (t) \wedge x \text{ in } [0,T] \times \R^3 ,
\\ \label{AM3}
  &  \mathcal{M}  \begin{bmatrix} \ell^E \\ r^E \end{bmatrix} '
   = (  b(u^E ,u^E , v_i) )_{i \in \{1,\ldots,6\}}  ,
\end{align}
where $ \mathcal{M} $ and $ v_i $ were introduced in Section \ref{regproperty}.

The vorticity equation \eqref{ExiOmega} can easily be inferred from \eqref{chEuler1},\eqref{chEuler2} and \eqref{LastReadingE} whereas
\eqref{AM3}  is obtained from \eqref{Tonga}, \eqref{chEulerweak} and some integration by parts.
Observe that \eqref{AM3} can be seen as the inviscid counterpart of \eqref{AM2}, and recall that  the $b(u^E ,u^E , v_i)$ can be computed thanks to the formula \eqref{ouf}.

\subsection{The infinite inertia limit}
\label{SectionInertiaEuler}

The following result  is the counterpart of Theorem \ref{NSBodyWeakMassless} for the inviscid system \eqref{chEuler1}-\eqref{chEulerSolideci}.
It shows that in  the limit of infinite inertia, that is, when $m$ and the eigenvalues $(\lambda_i )_{i =1,2,3} $ of $\mathcal{J}_0 $ converge to $+\infty$ with $\lambda_i = O(\lambda_j )$ for any $i,j$,
the system \eqref{chEuler1}-\eqref{chEulerSolideci} degenerates into the following classical Euler equations in $ \mathcal{F}_0$:
\begin{gather}
\label{chEuler1m}
\frac{\partial u^E }{\partial t} +   u^E \cdot\nabla  u^E  + \nabla p^E  =  0 \ \text{ for } \ x \in \mathcal{F}_0  , \\
\label{chEuler2m}
\div u^E  = 0 \ \text{ for } \  x \in \mathcal{F}_0  ,  \\
\label{chEuler3m}\
u^E  \cdot n = 0 \ \text{ for } \ x\in \partial \mathcal{S}_0  ,  \\
\label{chEulerci2m}
u^E  |_{t= 0} = u^E _0 .
\end{gather}
\begin{Theorem}
\label{EBodyWeakMassless}
Let be given  $\lambda \in (0,1)$ and $u^E_{0} \in  \mathcal{H}$ such that $\ell^E_{0} = r^E_{0} = 0$, $ u^E_{0} |_{{\mathcal{F}_0}} \in  H^{1} \cap C^{ 1,\lambda} $ and $\curl u^E_{0}|_{{\mathcal{F}_0}} $ is compactly supported.
Let be given  $\underline{m} >0$ and $\beta > 0$.

Then there exists $T>0$ such that for any  $m \geq \underline{m}$ and for any symmetric positive $3 \times 3$ matrix $ \mathcal{J}_0$ with eigenvalues $ (\lambda_i)_{i=1,2,3} $ satisfying  $\lambda_i \geq \beta$, the corresponding  solution $u^E$ of  \eqref{chEuler1}-\eqref{chEulerSolideci} given by Theorem \ref{EulerBodyStrong} is defined up to the time $T$.

Moreover in the infinite inertia limit,  $u^E \vert_{\mathcal{F}_0}$ converges  in $  L^{\infty}( 0,T;  C^{1,\tilde{\lambda}} (\mathcal{F}_0  )  )$, for any $\tilde{\lambda} \in (0,\lambda)$, to the unique smooth  solution of  \eqref{chEuler1m}-\eqref{chEulerci2m}  and $\ell^E $ and $r^E$ converge to $0$ in $C^{1}  ( 0,T  ; \R^{3})$.
\end{Theorem}
\begin{proof}
The key issue here is to obtain some estimates uniform in the infinite inertia limit. Let us stress in particular that, a priori, Theorem \ref{EulerBodyStrong} only provides, for some inertia $m$ and $\mathcal{J}_0$,  the existence of a smooth solution on a time interval $[0, T_{m,\mathcal{J}_0}]$ which may shrink when $m$ and the eigenvalues $(\lambda_i )_{i =1,2,3} $ of $\mathcal{J}_0$ go to infinity.
However we are going to see below that this is not the case. Actually revisiting  the proof of  Theorem \ref{EulerBodyStrong} given in  \cite{KatoBody} we will provide a time $T$ common to any large enough inertia.

The first basic observation in this direction is that according to the energy identity \eqref{EulerBodyStrongEnergy}, and because of the choice of the initial data (which are somehow ``well-prepared'') the energy $ \| u^E  (t, \cdot)  \|^2_{\mathcal{H}}$ does not depend on time nor on the inertia:
\begin{equation}
\label{eneunfi}
\int_{{\mathcal F}_{0}} | u^E  (t, \cdot) |^2 dx  + m \ell^{E} (t)\cdot \ell^{E} (t) + \mathcal{J}_{0} r^{E} (t)  \cdot  r^{E} (t)  = \int_{{\mathcal F}_{0}} | u^E _{0} |^2 dx  .
\end{equation}

Let us now show how to obtain a uniform estimate of the velocity in $L^{\infty}(0,T; C^{1,\lambda}({\mathcal F}_{0}))$ for some $T>0$, thanks to the vorticity formulation \eqref{ExiOmega}-\eqref{AM3}.

By standard transport estimates (cf. \cite{gamblin2}, Corollary 2.4) one infers from \eqref{ExiOmega} the estimate
\begin{equation}
\label{transU}
\begin{split}
\| \omega^E (t)\|_{ C^{0,\lambda}({\mathcal F}_{0}) } & \leq  \| \omega_0^E \|_{C^{0,\lambda}({\mathcal F}_{0}) }
\, \exp \Big( C \int_0^t \bigl(  \| u^E \|_{C^{1,\lambda}({\mathcal F}_{0}) } (s)  + \| r^E \| (s) \bigr) ds    \Big) \\
& \leq  \| u_0^E \|_{C^{1,\lambda}({\mathcal F}_{0}) } \exp(C\| r^E \|_{L^\infty(0,T)} t)
\, \exp \Big( C \int_0^t  \| u^E \|_{C^{1,\lambda}({\mathcal F}_{0}) } (s) ds      \Big),
\end{split}
\end{equation}
where $C$  does not depend on the inertia.

On the other hand by classical elliptic theory,  one infers from  \eqref{Eq:vtransporte} the following estimate, where time is omitted,
\begin{equation}
\label{EstLL}
\| u^E  \|_{C^{1,\lambda}({\mathcal F}_{0}) } \leq C  (     \| \omega^E \|_{C^{0,\lambda}({\mathcal F}_{0})} + \| \ell^E \| + \| r^E \| ) \leq
C ( \| \omega^E \|_{C^{0,\lambda}({\mathcal F}_{0})} +  \|  u^E_0 \|_{L^2 ({\mathcal F}_{0})}) ,
\end{equation}
with $C>0$ depending on  $\underline{m} >0$ and $\beta > 0$.

Plugging \eqref{EstLL} in \eqref{transU} yields,
\begin{equation*}
\| \omega^E (t)\|_{ C^{0,\lambda}({\mathcal F}_{0}) }
 \leq  \| u_0^E \|_{C^{1,\lambda}({\mathcal F}_{0}) } \exp(C\|  u^E_0 \|_{L^2 ({\mathcal F}_{0})} t)
\, \exp \Big( C \int_0^t  \| \omega^E \|_{C^{0,\lambda}({\mathcal F}_{0}) } (s) ds      \Big),
\end{equation*}
with $C>0$ depending on  $\underline{m} >0$ and $\beta > 0$, from which one deduces the existence of a small time $T>0$ and an estimate of  velocity in $L^{\infty}(0,T; C^{1,\lambda}({\mathcal F}_{0}))$ both uniformly for any  $m \geq \underline{m}$ and for any symmetric positive $3 \times 3$ matrix $ \mathcal{J}_0$ with eigenvalues $ (\lambda_i)_{i=1,2,3} $ satisfying  $\lambda_i \geq \beta$.

Let us now prove the second part of Theorem \ref{EBodyWeakMassless} about convergence of the fluid and solid velocities.
First  \eqref{eneunfi} yields that  $\ell^E $ and $r^E$ converge to $0$ in $L^{\infty}  ( 0,T  ; \R^{3})$ in the infinite inertia limit.
Then, similarly as we proceed in the proof of Theorem \ref{NSBodyWeakMassless}, by using \eqref{AM3} and \eqref{ouf} we write
\begin{equation*}
   \begin{bmatrix} \ell^E \\ r^E \end{bmatrix} '
   =  \mathcal{M}^{-1} (  \mathcal{T}_{1,i}^E)_{i \in \{1,\ldots,6\}}   + \mathcal{M}^{-1} \mathcal{T}_{2}^E  .
\end{equation*}
where,
\begin{equation*}
\mathcal{T}_{1,i}^E  := \int_{\mathcal{F}_0 } \Big(     [    \big( u^E -  u_\mathcal{S}^E )  \cdot\nabla   \big) \nabla \Phi_i] \cdot u^E   -  \det (r_u^E , u^E ,   \nabla \Phi_i ) \Big) ,
\text { and } \mathcal{T}_{2}^E :=
 \begin{bmatrix} m \, r^E \wedge \ell^E  \\ ( \mathcal{J}_0 r^E ) \wedge r^E \end{bmatrix} .
\end{equation*}
Thus, by using again the energy bound \eqref{eneunfi}, we deduce that   $\ell^E $ and $r^E$ converge to $0$ in $C^{1}  ( 0,T  ; \R^{3})$.

Finally it remains to prove that $u^{E} \vert_{\mathcal{F}_0}$ converges, up to a subsequence,  in $  L^{\infty}( 0,T;  C^{1,\tilde{\lambda}} (\mathcal{F}_0  ) )$, for any $\tilde{\lambda} \in (0,\lambda)$.
Since we have some uniform estimates of $u^{E} \vert_{\mathcal{F}_0}$ in the space $  L^{\infty}( 0,T;  C^{1,\lambda} (\mathcal{F}_0  )  )$, it is sufficient to obtain a temporal estimate uniform in  infinite inertia limit and to use the Aubin-Lions lemma to conclude that a subsequence is converging in $  L^{\infty}( 0,T;  C^{1,\tilde{\lambda}} (\mathcal{F}_0  )  )$, for any $\tilde{\lambda} \in (0,\lambda)$.
This can be easily achieved by using \eqref{chEulerweak} with some test functions $v$ compactly supported in $\mathcal{F}_0$.

These convergences are sufficient to yield that the limit is a smooth  solution of  \eqref{chEuler1m}-\eqref{chEulerci2m} on $[0, T]$.

Since classical solutions of the Euler equations are unique we finally deduce that the whole sequence is converging.
\end{proof}
\section{Inviscid limit}
\label{IL}
Let us now state the main result of this paper.
\begin{Theorem}\label{clasic}
The following holds true.
\begin{enumerate}
\item
With the notations of Theorem \ref{NSBodyWeak}  and Theorem \ref{EulerBodyStrong},
and assuming that $ u_0 $ converges to $ u^E_0 $ in $\mathcal{H}$ when $\nu$ tends to $0$ and that $\alpha := \alpha^\nu $ satisfies
 $\alpha \nu$ converges to $0$ when $\nu$ tends to $0$, then
 $ u $ converges to $ u^E $ in $L^\infty(0,T; \mathcal{H})$,  $ \sqrt{\nu} \| u  \vert_{{\mathcal{F}_0}}\|_{L^2(0,T;H^1 ({\mathcal{F}_0} ))}$ and
  $ \sqrt{\alpha \nu} \| u - u_\mathcal{S}  \|_{L^2( (0,T) \times \partial {\mathcal{S}_0} )}$ converge to $0$, where $T>0 $ is the lifetime of the smooth solution $u^E $ of the inviscid system
\item  Moreover $ (\ell,r) $ converges to $(\ell^E ,r^E )  $ in $H^{1} ( 0,T ;\R^{6} )$.
 \item Assuming  that $ u_0 = u^E_0 $ and that $\alpha >0$ does not depend on $\nu$, then
  there exists $C > 0 $ (depending on $T$) such that
\begin{equation}
\label{clasicesti2}
\|u -u^E\|_{L^\infty(0,T; \mathcal{H})}  + \sqrt{\nu} \|u -u^E\|_{L^2(0,T;H^1 ({\mathcal{F}_0} ))} \leq  C  ( 1 + \alpha ) \nu^{3/4}.
\end{equation}
\end{enumerate}
 \end{Theorem}

A few remarks are in order.
\par

Regarding the  first part of Theorem \ref{clasic},
let us first mention again the references \cite{coco,MasmoudiLSR,BGP} for the relevance of considering a friction coefficient depending on the viscosity, as a limit of accommodation boundary condition for the Boltzmann equation.
The first part of Theorem \ref{clasic} extends earlier results  where a fixed boundary was considered, see \cite{Iftimie-Planas,Paddick,BGP,xin}.
Let us point out in particular that it  could be possible to extend the first part of Theorem \ref{clasic} to a  weaker setting, following the analysis of \cite{BGP}.
This will require to extend P.-L. Lions' definition of dissipative solutions of the Euler incompressible equations to the system \eqref{chEuler1}-\eqref{chEulerSolideci},  and to modify the proof of the first part of Theorem \ref{clasic} below following a by now classical method, so-called the relative entropy method or the modulated energy method depending on the context and on the authors.
We choose here to consider smooth solutions of  the system \eqref{chEuler1}-\eqref{chEulerSolideci} in order to keep the unity of the theorem, since the other parts fail in a weaker context.
\par
\  \par

The second part of Theorem \ref{clasic} is perhaps the most surprising. It shows that
the convergence of the body's dynamics is better than the one implied by the convergence of  $ u $  to  $ u^E $ in the energy space $L^\infty(0,T; \mathcal{H})$ stated in the first part.
Indeed the latter provides a convergence  $ (\ell,r) $  to $(\ell^E ,r^E )  $ in $L^{\infty} ( 0,T ;\R^{6} )$. This last result relies on the possibility to compute explicitly, in the present case of the Navier slip conditions,  a well-known phenomenon in the theory of the systems involving an incompressible
flow and a structure, namely the added mass phenomenon, cf. for instance [6, 17].
\par
\  \par

If we focus on the dependence on the viscosity the estimate  \eqref{clasicesti2} says that $u $ converges strongly to $u^E$ in $L^\infty (0,T; L^2  ({\mathcal{F}_0})) $ with a rate of $O(\nu^{3/4})$ and in $L^2 (0,T; H^1  ({\mathcal{F}_0})) $ with a rate of $O(\nu^{1/4})$.  We therefore recover the optimal rate of convergence, with respect to $\nu$, found in \cite{Iftimie-Sueur} in the case where the Navier conditions are prescribed on a fixed boundary.

\par
\  \par

Theorem \ref{clasic} may be useful for controllability  issues, in particular for the global  controllability of the
system. In the case of a fixed boundary a strategy  initiated by Coron in \cite{coron} proved
the global approximate controllability for the 2-D incompressible Navier-Stokes equations with Navier slip boundary
conditions.
His proof relies on another of his earlier results about the global controllability of the incompressible  Euler  equations, see  \cite{coron2}.
This strategy was used later on by Chapouly in  \cite{chapouly} to extend Coron's result into a global null controllability result.

It is therefore possible that a similar strategy could be fruitful in the case of a moving rigid body. Note however that in such a case little is known so far about the controllability of the inviscid system.
Let us mention in that direction the result \cite{CM}  by Chambrion and Munnier in the irrotational case.

\section{Proof of Theorem \ref{clasic}}
\label{proof}

\subsection{Two basic observations}

Let us observe first that  when  $ u $ converges to $ u^E $ in $L^\infty(0,T; \mathcal{H})$ then it follows straightforwardly from \eqref{NSBodyWeakEnergy}, \eqref{EulerBodyStrongEnergy}
and Korn inequality that  $ \sqrt{\nu} \| u  \vert_{{\mathcal{F}_0}}\|_{L^2(0,T;H^1 ({\mathcal{F}_0} ))}$ and
  $ \sqrt{\alpha \nu} \| u - u_\mathcal{S}   \|_{L^2( (0,T) \times \partial {\mathcal{S}_0} )}$ converge to $0$.

Moreover the second part of Theorem \ref{clasic} is quite easy to obtain once the first part has been proved.
Indeed, using that
 $ u $ converges to $ u^E $ in $L^\infty(0,T; \mathcal{H})$, and that $ \sqrt{\nu} \| u \vert_{{\mathcal{F}_0}}\|_{L^2(0,T;H^1 ({\mathcal{F}_0} ))}$ and   $ \sqrt{\alpha \nu} \| u - u_\mathcal{S} \|_{L^2( (0,T) \times \partial {\mathcal{S}_0} )}$
 converge to $0$ as  $\nu$ tends to $0$, \eqref{AM2}, \eqref{AM3} and \eqref{ouf},
we obtain that
$  \mathcal{M}  \begin{bmatrix} \ell \\ r \end{bmatrix} ' $ converges to $ \mathcal{M}  \begin{bmatrix} \ell^E \\ r^E \end{bmatrix} ' $
in $L^2 (0,T)$. We then infer easily the second part.

\subsection{Proof of the first part}

Let us now turn to the proof of the first part.
In the following, $C$ will denote a constant independent of $\nu$
and $ \alpha$ that may change from one relation to another.

Let us also introduce the difference
 $$w:=u-u^E \text{ and similarly for the initial condition } w_0 := u_{0} -  u^E_{0} .$$
In the solid we will use the notation $$w_\mathcal{S} := u_\mathcal{S} - u^E_\mathcal{S} .$$
For any $t \in [0,T ]$, we have, thanks to \eqref{NSBodyWeakEnergy} and \eqref{EulerBodyStrongEnergy},
\begin{equation}
\nonumber \|  w(t,\cdot)    \|_{\mathcal{H}}^{2}  \leq  \| u_{0}  \|_{\mathcal{H}}^{2} + \|  u^E_{0}  \|_{\mathcal{H}}^{2} - 2 (u ,   u^E   )_{\mathcal{H}}  (t)  + 4 \nu  \int_{0}^{t} a(u,u) .
\end{equation}
We now apply \eqref{WeakNS}  to $v = u^E   $  to get
\begin{eqnarray*}
(u,  u^E  )_{\mathcal{H}}  (t) -  (u_{0},  u^E_{0} )_{\mathcal{H}}   &=&
 \int_{0}^{t} \Big[ (u, \partial_{t}   u^E )_{\mathcal{H}}     + 2 \nu a(u,u^E )  +
  b(u,u,u^E )     \Big] ds
  \\ &=& \int_{0}^{t} \Big[     2 \nu a(u,u^E )  +
  b(u,u,u^E )    - b(u^E,u,u^E )    \Big] ds  ,
\end{eqnarray*}
using \eqref{chEulerweak} with $v = u$.

Therefore,
\begin{eqnarray*}
\|  w(t,\cdot)    \|_{\mathcal{H}}^{2}
 &\leq& \|  w_{0}    \|_{\mathcal{H}}^{2} - 2  \int_{0}^{t} \Big[     2 \nu a(u,u^E )  - 2 \nu a(u,u ) +
  b(u,u,u^E )    - b(u^E,u,u^E )    \Big] ds
  \\ &&= \|  w_{0}   \|_{\mathcal{H}}^{2} - 2  \int_{0}^{t} \Big[  -   2 \nu a(u, w )  +
  b( w , w ,u^E )      \Big] ds  .
\end{eqnarray*}
This can be recast as follows:
\begin{equation}
\label{ola}
\|  w(t,\cdot)    \|_{\mathcal{H}}^{2}
+ 4 \alpha \nu \int_{0}^{t}  \int_{\mathcal{\partial S}_0}  | w - w_{\mathcal{S}}|^2
+4\nu  \int_{0}^{t}  \int_{\mathcal{F}_0}  |D(w)|^2
 \leq \|  w_{0}   \|_{\mathcal{H}}^{2} - 2  \int_{0}^{t} \Big[  -   2 \nu a(u^E , w )  +
  b( w , w ,u^E )      \Big] ds  .
\end{equation}

Therefore it only remains to use the Cauchy-Schwarz and Young inequalities, \eqref{bcontiE} and finally the Gronwall Lemma to achieve the first part of the theorem.

\subsection{Proof of the last part}

If one is interested in getting a better rate of convergence (for fixed $\alpha$) a further treatment of the viscous part of the right hand side
of  \eqref{ola} is needed. The underlying idea is somehow to put more derivatives on the inviscid solution $u^E$.

We use Lemma \ref{Useful} to obtain
\begin{equation*}
\begin{split}
a(u^E , w )   = &  -\alpha \int_{\partial \mathcal{S}_0}(u^E -  u^E_{\mathcal{S}} ) \cdot (w- w_{\mathcal{S}}) - \int_{\mathcal{F}_0}
D(u^E)\cdot D(w)
\\  = & - \alpha \int_{\partial \mathcal{S}_0}(u^E -  u^E_{\mathcal{S}} ) \cdot (w- w_{\mathcal{S}})
+ \frac{1}{2} \int_{ \mathcal{F}_0}  \Delta u^E \cdot w
 -  \int_{ \partial \mathcal{S}_0}    \Big( (D(u^E)  n)  \wedge n \Big)  \cdot \Big(  (w-w_\mathcal{S} ) \wedge  n \Big)
 \\ & -  \ell_{w} \cdot \int_{ \partial \mathcal{S}_0 } D(u^E) n \, ds    -  r_{w}  \cdot  \int_{ \partial \mathcal{S}_0 }  x \wedge D(u^E) n \, ds .
\end{split}
\end{equation*}

Thus we infer from \eqref{ola} the following inequality:
\begin{align*}
\|  w(t,\cdot)    \|_{\mathcal{H}}^{2} &
+4 \alpha \nu \int_{0}^{t}  \int_{\partial \mathcal{S}_0}  | w - w_{\mathcal{S}}|^2
+4\nu  \int_{0}^{t}  \int_{\mathcal{F}_0}  |D(w)|^2
 \\  \leq &
 \|  w_{0}    \|_{\mathcal{H}}^{2}
 - 4 \alpha \nu  \int_{0}^{t}  \int_{\partial \mathcal{S}_0} (u^E - u^E_\mathcal{S} ) \cdot   (w - w_{\mathcal{S}} )
 \\ &  +2 \nu \int_{0}^{t}  \int_{\mathcal{F}_0} \Delta u^E \cdot w
  -4\nu  \int_{0}^{t}  \int_{ \partial \mathcal{S}_0}    \Big( (D(u^E)  n)  \wedge n \Big)  \cdot \Big(  (w-w_\mathcal{S} ) \wedge  n \Big)
  \\ &  - 4 \nu  \int_{0}^{t}  \ell_{w} \cdot \int_{ \partial \mathcal{S}_0 } D(u^E) n \, ds
   - 4 \nu  \int_{0}^{t}  r_{w}  \cdot  \int_{ \partial \mathcal{S}_0 }  x \wedge D(u^E) n \, ds
  - 2  \int_{0}^{t}  b( w , w ,u^E )
  \\
    \leq &
 \|  w_{0}    \|_{\mathcal{H}}^{2} + \sum_{i=1}^6 I_i .
\end{align*}

To deal with $I_1$ and $I_3$  we will use the following lemma:
\begin{Lemma}\label{boundary}
There exists $C >0$ such that for any $\gamma >0$ and for any smooth function $f$ divergence
free and tangent to the boundary
\begin{equation*}
\| f \|_{L^2 (\partial \mathcal{S}_0 )}  \leq  C \gamma^{1/3} \| f \|_{L^2 ( {\mathcal{F}_0} )}^{2/3} + \frac{1}{4\gamma }  \| D( f) \|_{L^2 ( {\mathcal{F}_0} )}^2
+ C  \| f \|_{L^2 ( {\mathcal{F}_0} )}.
 \end{equation*}
\end{Lemma}

\begin{proof}
We first apply the following standard trace inequality:
\begin{equation*}
\| f \|_{L^2 (\partial \mathcal{S}_0 )}  \leq  C  \|f
\|_{L^2( {\mathcal{F}_0} )}^{1/2} \|  f \|_{H^1 ( {\mathcal{F}_0} )}^{1/2} ,
 \end{equation*}
  then the Korn inequality to find
\begin{equation*}
\| f \|_{L^2 (\partial \mathcal{S}_0 )}  \leq C( \|f \|_{L^2 ( {\mathcal{F}_0} )}^{1/2} \| D(f)
\|_{L^2 ( {\mathcal{F}_0} )}^{1/2} +\|f \|_{L^2 ( {\mathcal{F}_0} )} )
\end{equation*}
and finally Young's inequality to conclude.
\end{proof}
In order to apply this lemma we are going to substitute to the vector fields $  u^E_{\mathcal{S}}$ and $w_{\mathcal{S}}$ another divergence free vector field with the same traces on $\partial \mathcal{S}_0$ but with
a better decay at infinity, as we did in Section \ref{weaksolution}.
Indeed  let   $\chi$ be a smooth cut-off function defined on $\overline{\mathcal{F}_0}$ such that $\chi = 1$ in $\Gamma_{c}$ and  $\chi = 0$ in $\mathcal{F}_0 \setminus \Gamma_{2 c}$, where
\begin{equation*}
\Gamma_{c} := \{  x \in  \mathcal{F}_0 / \  d(x)  < c  \} \text{ with }  d(x) := dist (x,  \partial   \mathcal{S}_0 ) .
\end{equation*}
Let us denote
\begin{equation*}
    \psi^E_{\mathcal{S}} (t,x) :=  \frac12 (  \ell^E (t) \wedge  x  - r^E  (t)  | x |^{2} )   \text{ and } \tilde{u}^E_{\mathcal{S}} := \curl ( \chi  \psi^E_{\mathcal{S}} ) ,
\end{equation*}
and let us define similarly $\tilde{w}_{\mathcal{S}}$.
By Cauchy-Schwarz inequality we obtain
\begin{equation*}
 I_1 \leq
4\alpha\nu\int_0^t \|u^E- \tilde{u}^E_{\mathcal{S}} \|_{L^2(\partial \mathcal{S}_0)} \|w - \tilde{w}_{\mathcal{S}} \|_{L^2(\partial
\mathcal{S}_0)}  \leq
\alpha\nu C \int_0^t \|w - \tilde{w}_{\mathcal{S}} \|_{L^2(\partial
\mathcal{S}_0)}
\end{equation*}
and then applying Lemma  \ref{boundary} with $\gamma = \frac{\alpha C}{4}$
\begin{eqnarray*}
I_1  &\leq&  {\nu} \int_0^t \| D(w-\tilde{w}_{\mathcal{S}})
\|_{L^2( {\mathcal{F}_0} )}^2 + \alpha^{4/3}\nu C \int_0^t \|w-\tilde{w}_{\mathcal{S}}
\|_{L^2( {\mathcal{F}_0} )}^{2/3} +\alpha\nu C\int_0^t \|w - \tilde{w}_{\mathcal{S}} \|_{L^2( {\mathcal{F}_0} )}
\\  &\leq&  {\nu} \int_0^t \| D(w)
\|_{L^2( {\mathcal{F}_0} )}^2 + \nu C \int_0^t \|w  \|_{\mathcal{H}}^2 +  \alpha^{4/3}\nu C \int_0^t  \|w  \|_{\mathcal{H}} ^{2/3} +  \alpha\nu C \int_0^t \|w  \|_{\mathcal{H}}  .
\end{eqnarray*}

Now, thanks again to  Cauchy-Schwarz inequality and  applying Lemma  \ref{boundary} but this time with $\gamma = \frac{C}{4}$ we get
\begin{eqnarray*}
\nonumber I_3 & \leq & 4 \nu
\int_0^t \|D(u^E)n\|_{L^2(\partial \mathcal{S}_0)} \| w - \tilde{w}_{\mathcal{S}} \|_{L^2(\partial
\mathcal{S}_0)} \\ & \leq &  \nu C
\int_0^t \| w - \tilde{w}_{\mathcal{S}} \|_{L^2(\partial
\mathcal{S}_0)}
\\ \nonumber  & \leq& {\nu}\int_0^t\| D(w) \|_{L^2( {\mathcal{F}_0} )}^2 + \nu C \int_0^t \|w  \|_{\mathcal{H}}^2
+ \nu C \int_0^t
\|w \|_{\mathcal{H}}^{2/3} +  \nu   C  \int_0^t  \|w \|_{\mathcal{H}}.
\end{eqnarray*}
We simply estimate the following term by
$$I_2  \leq 2\nu \int_0^t \|\Delta u^E \|_{L^2( {\mathcal{F}_0} )} \|w \|_{L^2( {\mathcal{F}_0} )}  \leq   \nu C \int_0^t
\|w \|_{L^2( {\mathcal{F}_0} )}.$$
Finally it is straightforward that we can estimate the last  terms by
\begin{equation*}
I_4 + I_5 \leq\nu C \int_0^t \|w\|_{\mathcal{H} } \ \text{ and } \  I_6 \leq  C\int_0^t \|w\|_{\mathcal{H} }^2
\end{equation*}
thanks to \eqref{bcontiE}.

We deduce from the above relations  that
\begin{equation*}
 \|w(t)\|_{\mathcal{H}}^2  + 2 \nu \int_0^t \| D(w)
\|_{L^2( {\mathcal{F}_0} )}^2   \leq \beta(t) +  C\int_0^t\|w\|_{\mathcal{H}}^2 ,
\end{equation*}
with
\begin{equation*}
\beta(t) :=
 \|w_0\|_{\mathcal{H}}^2  +\nu C\int_0^t\|w\|_{\mathcal{H}}^2
+
 ( 1+ \alpha^{4/3} )\nu C \int_0^t \|w
\|_{\mathcal{H}}^{2/3}
 + (1+ \alpha )\nu C  \int_0^t \|w \|_{\mathcal{H}}  .
\end{equation*}
Since $ \|w(t)\|_{\mathcal{H}} \leq C$ we can deduce that
$$ \beta(t) \leq   \|w_0\|_{\mathcal{H}}^2  +(1 + \alpha +  \alpha^{4/3} ) \nu t C \|w
\|_{L^\infty(0,T;\mathcal{H})}^{2/3},$$
and finally, the Gronwall Lemma implies that
\begin{eqnarray*}
 \|w(t)\|_{\mathcal{H}}^2 & \leq & \|w_0\|_{\mathcal{H}}^2
+ (1 + \alpha +  \alpha^{4/3} )\nu  t  C \|w
\|_{L^\infty(0,T;\mathcal{H})}^{2/3}
 \\&& + C \int_0^t \Bigr(  \|w_0\|_{\mathcal{H}}^2
+  s (1 + \alpha +  \alpha^{4/3} ) \nu C\|w
\|_{L^\infty(0,T;\mathcal{H})}^{2/3}\Bigr)ds \exp (C t )
\\ & \leq &  C \Big(\|w_0\|_{ \mathcal{H}}^2 +
 ( 1+ \alpha + \alpha^{4/3}) \nu  \|w
\|_{L^\infty(0,T;\mathcal{H} )}^{2/3} \Big),
\end{eqnarray*}
where $C$ is a constant depending on $T$.

By setting  $ u_0 = u^E_0 $, we obtain estimate \eqref{clasicesti2}.

The proof of Theorem \ref{clasic} is then complete.
 \par
\par
\

 {\bf Acknowledgements.} The authors would like to express their gratitude to the
anonymous referee for his/her several suggestions which have improved
the paper.

 This work was partially supported by the CNRS-FAPESP cooperation. The first author thanks the
Laboratoire Jacques-Louis Lions - Paris 6, while the second author  thanks  the  IMECC-UNICAMP for their kind hospitality during their visits there. The first author was partially supported by CNPq-Brazil, grant 303302/2009-7 and  by FAPESP-Brazil, grant 2007/51490-7.
The second author was partially supported by the Agence Nationale de la Recherche, Project CISIFS,
grant ANR-09-BLAN-0213-02.  He also thanks Olivier Glass, Matthieu Hillairet and David G\'erard-Varet  for some fruitful discussions.

\end{document}